\documentclass{amsart}

\usepackage{amsmath}
\usepackage{amssymb}
\usepackage{amsthm}
\usepackage[all]{xypic}
\usepackage{fullpage}
\usepackage{verbatim}
\usepackage{setspace}
\usepackage[usenames]{color}
\usepackage[usenames]{color}
\usepackage{url}
\usepackage{hyperref}

\numberwithin{equation}{section}
\theoremstyle{plain}
\newtheorem{theorem}[equation]{Theorem}
\newtheorem{proposition}[equation]{Proposition}
\newtheorem{corollary}[equation]{Corollary}
\newtheorem{lemma}[equation]{Lemma}

\newtheorem{sublemma}[equation]{Sublemma}

\newtheorem{hypotheses}[equation]{Hypotheses}

\theoremstyle{definition}
\newtheorem{defn}[equation]{Definition}
\newtheorem{notation}[equation]{Notation}

\newtheorem{example}[equation]{Example}

\newtheorem{claim}[equation]{Claim}

\newcommand{\DMO}{\DeclareMathOperator}

\newcommand{\ff}{\footnote}

\newcommand{\beq}{\begin{equation}}
\newcommand{\eeq}{\end{equation}}

\newcommand{\bbar}[1]{\overline{#1}}

\newcommand{\dra}{\dashrightarrow}

\newcommand{\hra}{\hookrightarrow}

\newcommand{\ssm}{\smallsetminus}
\newcommand{\st}{\left\vert\right.}
\newcommand{\wt}{\widetilde}

\DeclareMathOperator{\Ann}{Ann}
\DeclareMathOperator{\Aut}{{Aut}}

\DeclareMathOperator{\Hom}{{Hom}}

\DeclareMathOperator{\id}{{Id}}
\DeclareMathOperator{\im}{Im}

\DeclareMathOperator{\Ker}{Ker}

\DeclareMathOperator{\Num}{Num}

\DeclareMathOperator{\red}{red}

\DeclareMathOperator{\sing}{sing}
\DeclareMathOperator{\Spec}{Spec}
\DeclareMathOperator{\Supp}{Supp}

\DeclareMathOperator{\rgr}{gr-\!}

\DeclareMathOperator{\rQgr}{Qgr-\!}

\newcommand{\sh}{\mathcal}

\newcommand{\mb}{\mathbb}

\newcommand{\CC}{{\mb C}}
\newcommand{\NN}{{\mathbb N}}
\newcommand{\PP}{{\mathbb P}}

\newcommand{\ZZ}{{\mathbb Z}}
\newcommand{\kk}{{\Bbbk}}

\newcommand{\sA}{\sh{A}}
\newcommand{\sB}{\sh{B}}

\newcommand{\sF}{\sh{F}}

\newcommand{\sH}{\sh{H}}
\newcommand{\sI}{\sh{I}}
\newcommand{\sJ}{\sh{J}}
\newcommand{\sK}{\sh{K}}
\newcommand{\sL}{\sh{L}}
\newcommand{\sM}{\sh{M}}
\newcommand{\sN}{\sh{N}}
\newcommand{\sO}{\sh{O}}
\newcommand{\sP}{\sh{P}}

\newcommand{\sR}{\sh{R}}
\newcommand{\sS}{\sh{S}}
\newcommand{\sT}{\sh{T}}

\newcommand{\sX}{\sh{X}}

\newcommand{\Ct}{\widetilde{C}}
\newcommand{\Vt}{\widetilde{V}}
\newcommand{\Xt}{\widetilde{X}}

\newcommand{\nba}{na\"ive blowup algebra}
\renewcommand{\H}{H}	
\DMO{\Pic}{Pic}

\title{Na\"ive blowups and canonical birationally commutative factors}
\author{T. A. Nevins}
\address{Department of Mathematics\\University of Illinois at Urbana-Champaign\\Urbana, IL 61801, USA}
\email{nevins@illinois.edu}
\author{S. J. Sierra}
\address{School of Mathematics\\University of Edinburgh\\King's Buildings, Mayfield Road\\Edinburgh EH9 3JZ, UK}
\email{s.sierra@ed.ac.uk}
\date{\today}


\begin{document}
\begin{abstract}
In 2008, Rogalski and Zhang  \cite{RZ2008} showed that if $R$ is a strongly noetherian connected graded algebra over an algebraically closed field $\kk$, then $R$ has a canonical birationally commutative factor.  This factor is, up to finite dimension, a twisted homogeneous coordinate ring $B(X, \sL, \sigma)$; here $X$ is the projective parameter scheme for point modules over $R$, as well as tails of points in $\rQgr R$.  (As usual,  $\sigma$ is an automorphism of $X$, and $\sL$ is a $\sigma$-ample invertible sheaf on $X$.)

We extend this result to a large  class of noetherian (but not strongly noetherian) algebras.   Specifically, let $R$ be a noetherian connected graded $\kk$-algebra, where $\kk$ is an uncountable algebraically closed field.  Let $Y_\infty$ denote the parameter space (or stack or proscheme) parameterizing $R$-point modules, and suppose there is a projective variety $X$ that  corepresents tails of points.  There is a canonical map $p:  Y_\infty \to X$.  If the indeterminacy locus of $p^{-1}$ is 0-dimensional and $X$ satisfies a mild technical assumption, we show that there is a  homomorphism $g:  R \to B(X, \sL, \sigma)$, and that $g(R)$ is, up to finite dimension, a naive blowup on $X$ in the sense of \cite{KRS, RS-0} and satisfies a universal property.
We further show that the point space $Y_\infty$ is noetherian.
\end{abstract}
\maketitle

\begin{center}
{\it Dedicated to Toby Stafford on the occasion of his 60th birthday, with gratitude and admiration}
\end{center}

\section{Introduction}\label{INTRO}
Let $\kk$ be an algebraically closed field, which we assume in the introduction to be uncountable. 
A foundational technique of noncommutative algebraic geometry is to study a finitely generated graded $\kk$-algebra $R= \kk \oplus R_1 \oplus R_2 \oplus \cdots $ via the parameter spaces for {\em point modules} over $R$:  these are cyclic (right) modules with the Hilbert series $1/(1-s)$.  The main idea is due to Artin, Tate, and Van den Bergh \cite{ATV1990}, although the fundamental result was proved by Rogalski and Zhang.

Rogalski and Zhang's result applies to strongly noetherian algebras.  
Recall that  $R$ is {\em strongly noetherian} if $R \otimes_\kk A$ is noetherian for any commutative noetherian $\kk$-algebra $A$.  
We recall also that if 
$X$ is a projective scheme,  $\sigma \in \Aut_\kk(X)$, and  $\sL$ is an invertible sheaf on $X$, then we may define a  {\em twisted homogeneous coordinate ring} $B(X, \sL, \sigma)$ by 
\[ B(X, \sL, \sigma) := \bigoplus_{n \geq 0} H^0(X, \sL \otimes \sL^\sigma \otimes \cdots \otimes  \sL^{\sigma^{n-1}}).\]
 (Here, as usual, we define $\sL^{\sigma}:= \sigma^* \sL$.)
If $\sL$ is appropriately ample (the technical term is {\em $\sigma$-ample}, defined in Section~\ref{DATA}), then $B(X, \sL, \sigma)$ is noetherian by \cite[Theorem~1.4]{AV}, and is strongly noetherian by \cite[Proposition~4.13]{ASZ1999}.  

Rogalski and Zhang's result is:

\begin{theorem}\label{thm-RZ}
{\rm (\cite[Theorem~1.1]{RZ2008})}
Let $R = \kk \oplus R_1 \oplus R_2 \oplus \cdots$ be a strongly noetherian graded algebra, generated in degree 1.  Then there is a $\kk$-algebra homomorphism $g: R \to B(X, \sL, \sigma)$, where $X$ is a projective scheme, $\sigma \in \Aut_\kk(X)$, and  $\sL$ is a $\sigma$-ample invertible sheaf on $X$.   Further,   $g$ is surjective in large degree and satisfies a universal property.
\end{theorem}

The universal property in the theorem means, roughly, that the kernel of $g$ annihilates any point module.   We prove that $g$ satisfies a different universal property in Theorem~\ref{thm-universal2}.  

The projective scheme $X$ in Theorem~\ref{thm-RZ} is the parameter scheme for point modules; it is a consequence of $R$ being strongly noetherian that point modules are parameterized by a scheme rather than a more exotic object \cite[Corollary~E4.5]{AZ2001}.    If $R$ is a (3-dimensional)  Sklyanin algebra, then $X$ is an elliptic curve and we recover the well-known embedding \cite{ATV1990} of an elliptic curve in a noncommutative $\PP^2$.  

Ring theory would be less interesting if all noetherian algebras were strongly noetherian.  Fortunately, this is not true, as shown by Rogalski \cite{R-generic}.  For example, let $X$ be a variety of dimension $\geq 2$ and let $ \sigma, \sL$ be as above.  Let $P \in X$,  and assume that the $\sigma$-orbit of $P$ is  {\em critically dense}:  that is, it is infinite, and any infinite subset  is Zariski-dense in $X$.  (More generally, we may take $P$ to be any 0-dimensional subscheme of $X$ supported at points with critically dense orbits.)    We may use these data to construct a subring
\[ R = R(X, \sL, \sigma, P) \subseteq B(X, \sL, \sigma),\]
by setting $R_n := H^0(X, \sI_P \sL( \sI_P \sL)^\sigma \cdots (\sI_P \sL)^{\sigma^{n-1}})$.  
The ring $R(X, \sL, \sigma, P)$ is called the {\em \nba}\ associated to the data $(X, \sL, \sigma, P)$. By \cite[Theorem~1.1]{KRS} and \cite[Theorem~1.1]{RS-0}, $R(X, \sL, \sigma, P)$ is always noetherian but never strongly noetherian, and thus Theorem~\ref{thm-RZ} does not apply. However,  there is trivially a map from $R(X, \sL, \sigma, P)$ to the overring $B(X, \sL, \sigma)$.  It is a bit embarrassing that current theory does not recover this inclusion.  

In this paper we remedy the embarrassment.    
Our results apply, however, to a much larger class of rings:  to all those whose point modules have the geometry of a \nba.   
Let $S$ be a noetherian connected graded algebra, generated in degree 1.  We seek to answer two questions.    First, is there a canonical map from $S$ to a twisted homogeneous coordinate ring?  
Second,  what can be said about the image of $S$ under this map?

To describe our main theorem, we establish notation for functors of points over noncommutative rings.  
Let  $R = \kk \oplus R_1 \oplus \cdots$ be a finitely generated graded algebra generated in degree 1.
We will denote the functor of 
(embedded) $R$-point modules by $F$.  That is,  if $A$ is a
commutative $\kk$-algebra, then $F( A)$ is the set of  isomorphism classes of graded cyclic right 
$R\otimes_{\kk} A$-modules $M$ such that each $M_n$ is a rank 1 projective
$A$-module. 
Let $M, N$ be $R$-point modules.  We say that 
$M \sim N$ if $M_{\geq n} \cong N_{\geq n}$ for $n \gg 0$; that is, $M$ and $N$ are equal as objects of noncommutative projective geometry, or, more formally,  of the category $\rQgr R$ of ``tails'' of graded modules (see \cite{AZ2001}).     
Let  $\H $ be the functor of ``tails of points''---that is, equivalence classes of modules whose Hilbert series is eventually equal to $1/(1-t)$ (see Section \ref{GENERAL}, page \pageref{tails of points} for a precise definition).
 Let $\pi:  F\to \H$ be the natural map.  

If we are willing to move beyond the category of schemes, the functor $F$ is always represented by a geometric object, which we call $Y_\infty$, and refer to as the {\em point space} of $R$; see Section~\ref{GENERAL}.  The geometry of $\H$ is often more complicated. For example, let $R(X, \sL, \sigma, P)$ be a \nba\ as above.    In an earlier paper \cite[Theorem~1.3]{NS1}, the authors proved that in this situation in fact  $ X$  {\em corepresents} the functor $F/\sim$ (in the category of finite type schemes), and we show here (Proposition~\ref{prop:existence of q}) that $X$ also corepresents $\H$.  In contrast, by \cite[Theorem~1.1]{KRS}, $\H$ is not represented by any scheme of finite type.

Our main result may be thought of as   a converse of sorts to  \cite[Theorem~1.3]{NS1}.  We show that, under a restriction on the geometry of point modules,   the existence of a  scheme $X$ that corepresents $\H$ ensures a canonical map from $R$ to a twisted homogeneous coordinate ring on $X$. Further,  we show that the image of $R$ is (up to finite dimension) a \nba.  More specifically, we have:

\begin{theorem}\label{thm-main}
Let $\kk$ be an uncountable algebraically closed field, and let $R= \kk \oplus R_1 \oplus \cdots $ be a  noetherian graded $\kk$-algebra generated in degree 1. Let $Y_{\infty}$
be the point space of $R$, as above.  
Suppose the following:
\begin{itemize}
 \item[(i)]
there is a commutative
diagram of morphisms
\[ \xymatrix@C=5pt{
& F \cong  Y_{\infty} \ar[ld]_{\pi} \ar[rd]^{p} \\
\H \ar[rr] && X, }\]
  where $X$ is  a projective scheme that corepresents $\H$ through the  morphism $\xymatrix{H \ar[r]&X.}$
\end{itemize}
Suppose further that:
\begin{itemize}
\item[(ii)]  $X$ is a variety of dimension $\geq 2$ that is either a surface or locally factorial at all points in the indeterminacy locus of $p^{-1}$;
\item[(iii)] the map  $\H \to X$ is bijective on $\kk$-points; 
\item[(iv)]  the indeterminacy locus of $p^{-1}$ 
consists  (set-theoretically) of
countably many points.
\end{itemize}
Then:
\begin{enumerate}
\item there are an automorphism $\sigma$ of $X$, a $\sigma$-ample invertible sheaf $\sL$ on $X$, and a $\kk$-algebra homomorphism $g: R \to  B(X, \sL, \sigma)$, universal for maps from $R$ to birationally commutative algebras; 
\item  there is a 0-dimensional subscheme $P$ of $X$, supported at points with critically dense orbits, so that the image of $g$ is equal (in large degree) to the \nba\ $R(X, \sL, \sigma, P)$.  
\end{enumerate}
We further obtain:
\begin{enumerate}
\item[(3)] the point space $Y_\infty$ is noetherian; and
\item[(4)]the indeterminacy locus of $p^{-1}$ is critically dense in $X$.
\end{enumerate}
\end{theorem}
\noindent
See Theorem \ref{thm-final} for a more detailed and technically precise statement of this result.

We show in  Corollary~\ref{cor:hypotheses apply to NBAs} that the hypotheses of   Theorem~\ref{thm-main} hold if $R = R(X, \sL, \sigma, P)$ is itself a \nba\ as above.  Our result thus recovers the inclusion of $R(X, \sL, \sigma, P) $ in $B(X, \sL, \sigma)$.  

The universal property in the theorem means, approximately, that any map from $R$ to a skew polynomial extension of a commutative $\kk$-algebra factors through $g$.
The existence of a factor of $R$ universal for maps from $R$ to birationally commutative algebras holds in substantial generality (see Theorem~\ref{thm-universal1}) and does not require any assumptions on the geometry of the point space $Y_\infty$.   

In \cite{NS1} we drew a parallel 
between $Y_\infty$ and the Hilbert scheme $S^{[n]}$ of $n$ points on $S$.  We claimed there that $Y_\infty$ should be thought of as a ``Hilbert scheme of one point on a noncommutative variety'' and that $X$ should be thought of as a ``coarse moduli space of one point,'' analogous to the symmetric product $\operatorname{Sym}^n(S)$.  In light of this analogy, it is perhaps not surprising that we are not certain whether condition (iii) of the theorem, although it appears highly plausible, is automatically satisfied for reasonable classes of algebras.  Indeed, 
the analogue for moduli of zero-dimensional sheaves  on a smooth commutative surface fails: it is well-known that $S^{[n]} \to \operatorname{Sym}^n(S)$ is a (nontrivial) resolution of singularities for $n > 1$.

We briefly describe the structure of the paper.
In Section~\ref{GENERAL} we  establish notation for point spaces of  algebras and define the point space $Y_\infty$ rigorously.  
In Section~\ref{NBEXAMPLE}, we describe the point space of a \nba\ and show that the hypotheses of Theorem~\ref{thm-main} apply.  
In Section~\ref{PPR} we show, under very weak conditions, that the point space gives rise to a factor of $R$ universal for maps to birationally commutative algebras.  
We believe this result is of independent interest. 
In Section~\ref{COARSE} we assume the existence of  a scheme $X$ that corepresents $\H$.  We construct both the automorphism $\sigma$ of $X$ and a map $g:  R \to \kk(X)[t; \sigma]$.
In Section~\ref{POINTS} we study how points and curves in $X$ lift  to $Y_\infty$.  We prove the key result  that any point where $p^{-1}$ is not defined must have infinite order under $\sigma$.  
In Section~\ref{DATA} we define the rest of the data for the \nba\ $R(X, \sL, \sigma, P)$.  We show that $\sL$ is invertible and $\sigma$-ample, and that the orbits of points in $P$ are dense.
Finally, in Section~\ref{PROOF} we prove Theorem~\ref{thm-main}.

\vspace{1em}

\noindent
{\bf Acknowledgments.}\; The authors are grateful to Hailong Dao, Dan Rogalski, Karl Schwede, Toby Stafford, Ravi Vakil, and Chelsea Walton for helpful conversations.  The first author was partially supported by NSF grants DMS-0757987 and DMS-1159468 and NSA grant H98230-12-1-0216.  The second author was partially supported by an NSF Postdoctoral Research Fellowship, grant DMS-0802935.  Both authors were supported by the NSF grant 0932078 000 at the MSRI program ``Noncommutative Algebraic Geometry and Representation Theory.'' 
 We are grateful to the referee of an earlier version of this paper for  many extremely useful comments. 

\vspace{1em}

\noindent
{\bf Notation.}  
The symbol $\kk$ will always denote an algebraically closed field, often but not always uncountable.  All algebras and schemes will be defined over $\kk$. Modules are by default right modules.
Our convention throughout is that if $f:  X \to Y$ is a morphism of schemes, then $f(X)$ denotes the  scheme-theoretic image of $X$ unless otherwise stated.

\section{Generalities}\label{GENERAL}
In this section, we place the geometry of point schemes on a firmer footing and define the point space $Y_\infty$ of an algebra rigorously. 

A {\em connected graded} $\kk$-algebra is an $\NN$-graded $\kk$-algebra $R= R_0 \oplus R_1 \oplus \cdots$, with $R_0 =\kk$ and all $R_i$ finite-dimensional. 
Let 
$R$ be a connected graded  $\kk$-algebra generated in degree 1.  We use subscript notation to denote base extension; that is, if  $A$ is a commutative $\kk$-algebra, let $R_A:= R\otimes_\kk A$.  Recall that a graded $R$-module $M$ is {\em bounded}  if there is an $n_0$ such that $M_n=0$ for $n\geq n_0$.  If all finitely generated submodules of a module $M$ are bounded, we say that $M$ is {\em torsion}.  If  $M$ has no nonzero bounded submodules, then $M$ is {\em torsion-free}.

An {\em $R_A$-point module} is a graded quotient $M$ of $R_A$ so that  $M_i$ is rank 1 projective over $A$ for all $i \geq 0$; equivalently, $M$ is a graded cyclic $R_A$-module, flat over $A$,  with Hilbert series $1/(1-s)$.  There is a (covariant) functor of points $F:  \mbox{ Commutative $\kk$-algebras } \to \mbox{ Sets }$ given by 
\[F(A) = \{ \mbox{ isomorphism classes of $R_A$-point modules } \}.\]
  We sometimes, without comment, treat $F$ as a  contravariant functor from (affine) Schemes  to Sets.

The functor $F$ is  a Hilbert functor as in \cite{AZ2001}, associated to the Hilbert series $1/(1-s)$.  If we consider the equivalent functor for the Hilbert series $1+s+\cdots+s^n$, it is clear that it is represented by a projective scheme.  Let this scheme be $Y_n$  (for $n \in \NN$).    That is, $Y_n $ parameterizes cyclic $R$-modules $M = M_0 \oplus \cdots \oplus M_n$, where $\dim M_i = 1 $ for $0 \leq i \leq n$.  Note that $Y_0 = \Spec \kk$.

 The truncation map $M \mapsto M/M_n$ gives a morphism $\phi_n:  Y_n \to Y_{n-1}$, and the functor of points $F$ is represented in abstract terms by   
$\displaystyle Y_{\infty} := \varprojlim_{\phi_n} Y_n$.
We refer to $Y_\infty$ as the {\em point space} of $R$. 

  In \cite{NS1}, we considered $Y_{\infty}$ as an {\em fpqc-stack}.  
  Since automorphism groups of closed points are trivial, the terminology {\em fpqc-space} may be better, and we use it in the current paper.

  We give the definitions of these concepts:
  \begin{defn}\label{def-space}(cf.  \cite[Definition~4.1]{NS1})
  An {\em fpqc-space} is a functor $\sh X:    {\rm Schemes }  \to  {\rm  Sets }$ that is a sheaf in the fpqc topology (note that, in particular, objects have no nontrivial automorphisms). 
   If the diagonal 
$\Delta: {\mathcal X}\rightarrow {\mathcal X}\times_{\kk}{\mathcal X}$ is representable, separated, and quasi-compact, and if $\sh X$  admits
a representable surjective fpqc morphism $Z\rightarrow {\mathcal X}$ from a scheme $Z$, then $\sh X$ is {\em fpqc-algebraic}.  
An fpqc-algebraic space 
 $\sX$ is {\em noetherian}  if it admits an fpqc atlas $Z \rightarrow {\mathcal X}$ by a
noetherian scheme $Z$.  
\end{defn}

We trivially have:

\begin{proposition}\label{prop:space}
Let $\dots \to X_{n+1} \to X_n \to \dots$ be a system of projective schemes and morphisms. Then $X_\infty = \varprojlim X_n$ is  an fpqc-space.
\end{proposition}
\begin{proof}
This follows since the sheaf $X_\infty$ is the limit of sheaves of sets, hence is itself a sheaf of sets.
 \end{proof}

In particular, Proposition~\ref{prop:space} shows that the point space $Y_\infty$  is an fpqc-space.

To give a morphism from  a scheme $W$ to $Y_\infty$ is the same as giving  a collection of compatible maps $W \to Y_n$.  
A morphism from $Y_{\infty}$ to $W$ is defined as a morphism of
functors from $Y_{\infty}$ to $h_W$, the functor of points of $W$.    For example, there is a natural  morphism 
$\Phi_n:  Y_{\infty} \to Y_n$ that sends  a point module $M$ to $M/M_{\geq n+1}$.

We may also consider $Y_\infty$  as a proscheme; this is the perspective of \cite{AZ2001}.  However, we will see that these two points of view are equivalent.

\begin{proposition}\label{prop-surj-factor}
Let $\cdots Y_{n+1} \rightarrow Y_n \rightarrow \dots \rightarrow Y_1$ be a
system of projective $\kk$-schemes and proper and scheme-theoretically surjective morphisms
(that is, proper morphisms  $Y_{n+1}\rightarrow Y_n$ with scheme-theoretic image
$Y_n$).  Let 
$\displaystyle Y_\infty := \lim_{\longleftarrow} Y_n$ be the fpqc-space defined as above. 
 Suppose $p: Y_\infty
\rightarrow X$ is a morphism to a scheme $X$ of finite type over $\kk$.  Then
there is a compatible system of factorizations of the morphism $p$ through the
schemes $Y_n$ for all $n$ sufficiently large.  
\end{proposition}
\begin{proof}
 First note that the maps $Y_{n+1}\to Y_n$ are all set-theoretically surjective.
Fix $n_0\geq 1$.  Choose a finite
cover of $Y_{n_0}$ by affine schemes $U_{n_0, \alpha}$ and let $U_{n_0} :=
\coprod U_{n_0,\alpha}$. Now, given a cover $\{U_{n,\alpha}\}$ of $Y_n$ by affines, 
let $\{U_{n+1,\alpha}\}$ be a finite affine cover subordinate to the cover of $Y_{n+1}$ by
the inverse images of the $U_{n,\alpha}$; then, defining 
$U_{n+1} = \coprod U_{n+1,\alpha}$, we get a commutative diagram 
\begin{displaymath}
\xymatrix{Y_{n+1} \ar[d] & U_{n+1} \ar[d]\ar[l]\\
Y_n & U_n \ar[l]}
\end{displaymath}
with  set-theoretically and scheme-theoretically surjective vertical arrows.  Since each
$U_n$ is a finite disjoint union of affine schemes it is itself affine, say $U_n
= \Spec(T_n)$, and since
$U_{n+1}\rightarrow U_n$ is   scheme-theoretically surjective we get
a system of injective ring homomorphisms $T_n \hookrightarrow T_{n+1}
\hookrightarrow \dots$.  
Let $T := \displaystyle \lim_{\longrightarrow} T_n$ and $U _\infty:= \Spec(T)$. 
Then, since $\displaystyle U_\infty  = \lim_{\longleftarrow} U_n$, we get a
morphism
$U_\infty \rightarrow Y_\infty$ making all the squares  
\begin{displaymath}
\xymatrix{Y_{\infty} \ar[d] & U_\infty \ar[d]\ar[l]\\
Y_n & U_n \ar[l]}
\end{displaymath}
commute.  

Since $X$ is of finite type, the composite $U_\infty\rightarrow Y_\infty\rightarrow X$
factors through a finite-type affine scheme $W = \Spec T'$ for some finite-type subalgebra $T'$ of $T$.  As $T$ is the 
union of the $T_n$, we have $T' \subseteq T_n$ for $n \gg 0$.  It follows that $U_\infty\rightarrow X$ factors through 
morphisms $U_n\rightarrow X$ for all $n$ sufficiently large.  

We now use (Zariski) descent to obtain a morphism $Y_n\rightarrow X$ for $n\gg
0$.  To show that $U_n\rightarrow X$ is the composite $U_n\rightarrow
Y_n\rightarrow X$ with a morphism $Y_n\rightarrow X$, it suffices to check that
the composites $U_n\times_{Y_n} U_n \rightarrow U_n\rightarrow X$ with the two
projections $\Pi_i: U_n\times_{Y_n} U_n \rightarrow U_n$ ($i=1,2$) coincide.  By
hypothesis each $Y_n$ is projective, hence separated.  It follows that each
fiber product $U_n\times_{Y_n} U_n$ is again an affine scheme, say
$U_n\times_{Y_n} U_n =\Spec(\tilde{T}_n)$.  Moreover, the map
$U_{n+1}\times_{Y_{n+1}} U_{n+1} \rightarrow U_n \times_{Y_n}U_n$ is set-theoretically and
scheme-theoretically surjective for each $n$ (each of these is a disjoint union
of intersections of the open sets in the affine open covers of $Y_{n+1}$,
respectively $Y_n$, so this follows from the  surjectivity of
$Y_{n+1}\rightarrow Y_n$ since the cover of $Y_{n+1}$ refines the cover of
$Y_n$).  

Observe next that 
$\displaystyle \lim_{\longleftarrow}\, U_n \times_{Y_n} U_n = U_\infty
\times_{Y_\infty} U_\infty$ (cf. Lemma 4.5 of \cite{NS1}).  Because the map
$U_\infty\rightarrow X$ is defined as the composite 
$U_\infty\rightarrow Y_\infty\rightarrow X$, the two maps
$U_\infty\times_{Y_\infty} U_\infty \xrightarrow{\Pi_i} U_\infty \rightarrow X$,
where $\Pi_i$ 
the projection on the $i$th factor, coincide.  But we have:
\begin{lemma}\label{lem-surj-equal}
\mbox{}
\begin{enumerate}
\item
Suppose $f, g: Z\rightrightarrows X$ are two morphisms of schemes and $\Pi: Z'\rightarrow
Z$ is set-theoretically and scheme-theoretically surjective.  Then $f=g$  if and only if
$f \Pi = g\Pi$. 
\item 
\label{scholium}
Let $f:  C \to D$ and $g: D \to C$ be morphisms of schemes, where $f$ is set-theoretically and scheme-theoretically surjective and $gf = \id_C$.  
Then $f$ and $g$ are inverse isomorphisms. 
\end{enumerate}
\end{lemma}
\begin{proof}[Proof of Lemma]
(1)
By hypothesis, $\Pi$ is surjective as a map of topological spaces.  Hence $f=g$ as maps of topological spaces if and only if $f\Pi = g\Pi$ as maps 
of topological spaces.  Now suppose that $f\Pi = g\Pi$ as morphisms of schemes; thus $f=g$ as maps of topological spaces.  We  thus have $(f\Pi)_* \sO_{Z'} = (g\Pi)_* \sO_{Z'}$.  
Scheme-theoretic surjectivity of $\Pi$ means that 
$\mathcal{O}_Z\rightarrow \Pi_*\mathcal{O}_{Z'}$ is injective; thus, so are $f_*\mathcal{O}_Z\rightarrow (f\Pi)_*\mathcal{O}_{Z'}$ and $g_*\mathcal{O}_Z\rightarrow (g\Pi)_*\mathcal{O}_{Z'}$.  

Let $\sA := (f\Pi)_* \sO_{Z'} = (g\Pi)_* \sO_{Z'}$.
Abusing notation, we write $f_* \sO_Z, g_* \sO_Z \subseteq \sA$.
There are sheaf maps $f^{\#}:  \sO_X \to f_* \sO_Z$, $g^{\#}: \sO_X\to g_* \sO_Z$.
Since $f \Pi = g \Pi$, for any open set $U$ of $X$ and $a \in \sO_X(U)$ we have $f^{\#}(a) = g^{\#}(a)$ as elements of $\sA(U)$.  
Thus $f=g$.

(2)
We have $f =fgf$.  By part (1), $fg = \id_D$.  The  statement follows immediately.
\end{proof}
Since $U_\infty\times_{Y_\infty} U_\infty \rightarrow U_n\times_{Y_n} U_n$
is an inverse limit of set-theoretically and scheme-theoretically surjective maps of affine
schemes, and hence is itself set-theoretically and scheme-theoretically surjective,
it follows that the 
composite maps $U_n\times_{Y_n} U_n\xrightarrow{\Pi_i} U_n \rightarrow X$
coincide.   So the map $U_n\rightarrow X$ descends 
to a map $Y_n\rightarrow X$, as desired.
\end{proof}

We can now show that the proscheme and fpqc-space points of view are equivalent.

\begin{corollary}\label{cor-stack=proscheme}
Let $\cdots Y_{n+1} \stackrel{\phi_{n+1}}{\rightarrow} Y_n \rightarrow \dots \rightarrow Y_1$ 
and $\cdots X_{n+1} \rightarrow X_n \rightarrow \dots \rightarrow X_1$ 
be 
systems of  projective schemes and  morphisms.  Let 
$\displaystyle Y_\infty = \lim_{\longleftarrow} Y_j$ and $\displaystyle X_\infty = \lim_{\longleftarrow} X_i$ be the fpqc-spaces defined above.   
Let $Y'_m$ be the  smallest closed subscheme of $Y_m$ through which the natural map $Y_\infty\to Y_m$ factors.
Then 
\[ \Hom(Y_\infty, X_\infty) = \varprojlim_i \varinjlim_j \Hom(Y'_j, X_i).\]
\end{corollary}

Before proving this, we give two easy but useful lemmas.

\begin{lemma}\label{lem:chain}
 Let $Z_1 \supseteq Z_2 \supseteq \dots$ be an infinite descending chain of closed subschemes of a noetherian scheme $X$, and let $\phi:X\to Y$  be a morphism of schemes.
Then $\phi(\bigcap_i Z_i) = \bigcap_i \phi(Z_i)$.
\end{lemma}
\begin{proof}
 As $X$ is noetherian,   $Z_k = \bigcap_i Z_i$ for some $k$.  But then we have:
\[ \phi(Z_k) = \phi(\bigcap_i Z_i) \subseteq \bigcap_i \phi(Z_i) \subseteq \phi(Z_k).\]
\end{proof}

\begin{lemma} \label{lem:Yprime}
Assume the hypotheses and notation of Corollary~\ref{cor-stack=proscheme}.  
 For $m \in \NN$,  let $Y_m'' = \bigcap_{j \geq 0} \phi^j(Y_{m+j})$, where as usual $\phi^j(Y_{m+j})$ means the scheme-theoretic image.  
Then $Y_m' = Y_m''$ for all $m$.
\end{lemma}
\begin{proof}
Let $\Phi_m:  Y_\infty \to Y_m$ be the natural map.
For any $j \in \NN$, we have $Y'_m = \Phi_m(Y_\infty)= \phi^j \Phi_{m+j} (Y_\infty) \subseteq \phi^j(Y_{m+j})$.  
Thus $Y'_m \subseteq Y''_m$.  
For the other direction, by Lemma~\ref{lem:chain}, we have $\phi(Y''_{m+1}) = \bigcap_{j \geq 1} \phi^j(Y_{m+j}) = Y''_m$.
Thus $\dots \to Y''_{m+1} \stackrel{\phi}{\to} Y''_m \stackrel{\phi}{\to} \dots$ is a system of scheme-theoretically surjective proper morphisms.
It is immediate that $\Phi_m(\varprojlim(Y''_j)) = Y''_m$.  
As $\varprojlim Y''_j \subseteq Y_\infty$ we have $Y''_m \subseteq \Phi_m(Y_\infty) = Y'_m$.
\end{proof}

\begin{proof}[Proof of Corollary~\ref{cor-stack=proscheme}]
By definition, 
\[ \Hom(Y_\infty, X_\infty) = \varprojlim_i \Hom (Y_\infty, X_i).\]
Thus it suffices to prove the corollary in the case that $X_\infty = X_i$ is a projective scheme.

For all $n \geq m \in \NN$, we write $\phi^{n-m}:  Y_{n} \to Y_m$ for the induced morphism.  Now, it is easy to see that $Y_\infty = \varprojlim Y'_m$. 
By Lemma~\ref{lem:Yprime}, we have $Y'_m = \bigcap \phi^j(Y_{m+j})$, and it follows from Lemma~\ref{lem:chain} that $\phi:Y'_{m+1} \to Y'_m$ is (scheme-theoretically) surjective.
 That is, we may replace $Y_m$ by $Y'_m$ and assume without loss of generality that the $\phi_k$ are all scheme-theoretically surjective.  But now it follows directly from Proposition~\ref{prop-surj-factor} that $\Hom(Y_{\infty}, X_i) = \varinjlim_j \Hom(Y'_j, X_i)$.  
\end{proof}

Note that 
$\displaystyle \Hom_{\rm{proscheme}} (Y_\infty, X_\infty) = \varprojlim_i \varinjlim_j \Hom(Y_j, X_i)$
(see \cite[Definition~A.2.1]{etalehomotopy})  Thus Corollary~\ref{cor-stack=proscheme} implies that  we  may equivalently think of $X_\infty$ and $Y_\infty$ as proschemes or as fpqc-spaces, as long as we are willing to replace $Y_j$ by $Y'_j$.

In the sequel, given a morphism $f:  Y \to X = \varprojlim X_n$, we let $f_n: Y \to X_n$ be the composition $Y \to X \to X_n$ without comment.
Motivated by Corollary~\ref{cor-stack=proscheme}, we make the following definition:
\begin{defn}\label{def:image}
Let $Y$ be an fpqc-algebraic space and $X$ a scheme, and let $f:Y \to X$ be a morphism.  
The {\em scheme-theoretic image of $f$} is the smallest closed subscheme of $X$ through which $f$ factors; we write it as $f(Y)$, in keeping with our usual convention.
Now suppose that   $X= \varprojlim X_n$, where the $X_n$ are schemes. 
We define $f(Y) =  \varprojlim f_n(Y)$.  
\end{defn}

Throughout the paper, we will  use the following notation for  parameter spaces of point modules. 
\begin{notation}\label{not1}
Let $R$ be a connected graded  $\kk$-algebra generated in degree 1.  Let    $Y_n$  (for $n \in \NN$) be the projective scheme representing graded cyclic $R$-modules with Hilbert  series $1 + s + \cdots + s^n$. 

If $n \geq 1$, there are two morphisms $\phi_n, \psi_n:  Y_n \to Y_{n-1}$, where $\phi_n(M) =
M/M_n$ and  $\psi_n(M) = M[1]_{\geq 0}$.  
(Here $[1]$ denotes degree shift.)
 Clearly $\phi_{n-1} \psi_n = \psi_{n-1} \phi_n$.  
We will write $\phi^k:  Y_{n+k} \to Y_n$ to denote the map
$\phi_{n+1} \phi_{n+2}\cdots \phi_{n+k}$, and similarly for $\psi^k$.

Let $Y_\infty:= \varprojlim_\phi Y_n$.  We refer to $Y_\infty$ as the the {\em point space} of $R$, and to the collection $\{ Y_n, \phi_n, \psi_n \}$ as the {\em point scheme data} of $R$.  
Let $\Psi:  Y_{\infty}\to
Y_{\infty}$ be the endomorphism of $Y_{\infty}$ induced from the maps $\psi_n$;
that is, 
$\Psi(M) =  M[1]_{\geq 0}$.    Let  $\Phi_n:  Y_{\infty} \to Y_n$ be the canonical morphism.

The fundamental category of noncommutative projective geometry is the category $\rQgr R$ of graded $R$-modules modulo torsion $R$-modules, where  a module is {\em torsion} if it is equal to the direct limit of its right bounded submodules.  We will say that $M \sim N$ for two graded modules $M, N$ if their equivalence classes are equal in $\rQgr R$.  

Recall the definition of $F \cong
Y_\infty$, the  functor of embedded point modules of $R$, from the Introduction.     
Let $G = F/\sim$ and let $\H$ be the functor of {\em tails of points}.\label{tails of points} 
More precisely, if $A$ is a commutative $\kk$-algebra, let $\H(A)$ be the set of $\Pic(A)$-orbits of
\begin{multline*}\big\{ \mbox {$\rQgr R_A$-isoclasses of finitely generated graded $R_A$-modules $M$}\\ \mbox{with $M_n$ a rank 1 projective $A$-module for $n \gg 0$} \big\}.\end{multline*}

If $A$ is a commutative $\kk$-algebra, an element of $\H(A)$ is an equivalence class of $R_A$-modules; if $M$ is a representative of the equivalence class we write $[M] \in \H(A)$.
There are  canonical morphisms $F \to G \to \H$; we let  $\pi:  F \to H$ be the composition. Of course $\pi M = [M]$.
\end{notation}
\begin{proposition}
The functor $\H$ is a sheaf in the fpqc topology. 
\end{proposition}

\begin{proof}
Let $\alpha:  A \to B$ be a faithfully flat homomorphism of commutative $\kk$-algebras, and write $d^0, d^1: B \to B\otimes_A B$ for the natural maps.
Recall that $R_A := R \otimes_\kk A$.  Given an $R_A$-module $M$, let $M_B := M \otimes_A B$.
   We must show that
\begin{equation}\label{sequence}
\H(A) \stackrel{\alpha}{\to} \H(B) \rightrightarrows \H(B\otimes_A B)
\end{equation}
is an equalizer.  

First, suppose that $M$ and $M'$ are graded $R_A$-modules representing classes $[M], [M']$  in $\H(A)$ that become equal in $\H(B)$.  Since $M, M'$ are finitely generated and $R$ is generated in degree 1, there exists $n$ such that:
\begin{enumerate}
\item  $M_{\geq n}$ and $M'_{\geq n}$ are $R_A$-modules generated in degree $n$;
\item  $M_{n'}$ and $M'_{n'}$ rank one projective for $n' \geq n$;
\item there is  a rank one projective $B$-module $P$ for which $(M_B)_{\geq n} \otimes_B P \cong (M'_B)_{\geq n}$ as graded $R_B$-modules.  
\end{enumerate}
We may act by $\Pic(A)$ on $M$ and $M'$ to obtain without loss of generality that $M_n \cong A \cong M'_n$; this now forces $P \cong B$ and $(M_B)_{\geq n} \cong (M'_B)_{\geq n}$.
 
Since $M, M'$ are generated in degree $n$, there are  surjective $R_A$-module maps $f:  R_A[-n] \to M_{\geq n}$ and $g:  R_A[-n] \to M'_{\geq n}$.  
Since $(M_B)_{\geq n} \cong (M'_B)_{\geq n} $ is an isomorphism, we have $\Ker f_B = \Ann_{R_B}((M_B)_n) = \Ann_{R_B}((M'_B)_n) = \Ker g_B$.  Faithful flatness of $B_A$ gives that 
$ \Ker f = \Ker f_B \cap R_A = \Ker g_B \cap R_A = \Ker g$, and $M_{\geq n} \cong M'_{\geq n}$.  This  shows that $\alpha: \H(A) \to \H(B)$ is injective.

Now suppose that $M\in H(B)$ has the property that $[d^0(M)] = [d^1(M)]$: that is, there are a rank one projective $B\otimes_A B$-module $P$ and an isomorphism $d^0(M)_{\geq n}\otimes_{B\otimes_AB} P \rightarrow d^1(M)_{\geq n}$.  Again, choose $n$ sufficiently large so that $M_{\geq n}$ is generated in degree $n$ and $M_n$ is a rank one projective $B$-module.  Replacing $M$ by $M\otimes_B M_n^{-1}$ (which does not change the corresponding element of $H(B)$), we may assume $M_n\cong B$, and thus $d^0(M)_n \cong B\otimes_A B \cong d^1(M)_n$; in particular, now $P$ must be the trivial module, and we get an isomorphism $d^0(M)_{\geq n}\cong d^1(M)_{\geq n}$.  
It now follows   as in the proof of \cite[Lemma~E3.5]{AZ2001} that $M' := \ker (d^0 -d^1) \subseteq M$ determines an element of $H(A)$ and $M'_{\geq n} \cong M_{\geq n}$.  Thus \eqref{sequence} is exact at $H(B)$.  
\end{proof}

We note that there is an alternative way to define $\H$.

\begin{lemma}\label{lem:Hdefn}
 The functor $\H$ is the colimit of the system
\beq\label{Psisys} \xymatrix{\dots \ar[r] &F \ar[r]^{\Psi} & F \ar[r]^{\Psi} & \dots}.\eeq
In particular, $\H$ is a {\em cocone} of \eqref{Psisys}: that is, there are morphisms $\pi_n: F \to H$ so that $\pi_0 = \pi$ and $\pi_{n} = \pi_{n+m} \Psi^m$ for all $n, m \in \NN$.
\end{lemma}

\begin{proof}
Trivially, $\Psi: [M] \mapsto [M[1]]$ is  an automorphism of $\H$.
Let $\pi_n = \Psi^{-n} \pi$.  Then  $\pi_n \Psi^{n-m} = \pi_m$ as needed, since $\Psi\pi = \pi \Psi$.

 We need to show that $\H$ is universal among cocones of \eqref{Psisys}.  So let $H':  \{ \mbox{ commutative $\kk$-algebras } \} \to \{ \mbox{Sets} \}$
be a cocone:  that is, assume there are morphisms $h_n:  F \to H'$ with 
\beq \label{starstar} h_{n+k} \Psi^k = h_n. \eeq
We construct $\alpha:  \H \to H'$.

Let $A$ be a commutative noetherian $\kk$-algebra, and let $M \in \rgr R_A$ be a finitely generated $R_A$-module so that $M_n$ is a rank 1 projective $A$-module for $n \geq n_0$.  
As $R$ is generated in degree 1, there is $k \geq n_0$ so that $M_{\geq k} = M_k R_A$.  
Let $M' = (M_k)^{-1} \otimes_A M_{\geq k}[k]$.  
Note that $M'$ is an  $R_A$-point module, so $M' \in F(A)$, and  $[M'] =  \pi M' = \Psi^k [ M] $.  We define $\alpha[ M] = h_k(M')$.  
By \eqref{starstar} we have $h_{j+k} \Psi^{j} (M') = h_k (M')$, so $\alpha$ does not depend on $k$ and is well-defined. Further, $\alpha$ is clearly functorial in $A$.  
By construction, $\alpha \pi_j = h_j$ for all $j \in \NN$.
Uniqueness of $\alpha$ is immediate.
\end{proof}

The final result of this section, which we will use  often, is due to  Rogalski and Stafford.

\begin{proposition}\label{prop-RS}
{\rm (\cite[Corollary~2.7]{RS})}
Fix $N \in \NN$.  Let $R$ be a noetherian connected graded $\kk$-algebra generated in degree 1, and adopt Notation~\ref{not1}.  For each $n \geq N$, let $L_n$ be a finite set of (not necessarily closed) points of $Y_n$, such that $\phi_n(L_n) \subseteq L_{n-1}$, or respectively $\psi_n(L_n) \subseteq L_{n-1}$.  Then for $n \gg N$, the cardinality $\# L_n$ is constant, and $\phi_n$, or respectively $\psi_n$,  gives a bijection between $L_n$ and $L_{n-1}$.  Moreover, if $x \in L_{n-1}$ for any such $n$, then $\phi_{n}^{-1}$, or  respectively $\psi_{n}^{-1}$, is defined and is a local isomorphism at $x$.  \qed
\end{proposition}

\section{Point spaces of na\"ive blowups}\label{NBEXAMPLE}
In this section, we apply the constructions in Section~\ref{GENERAL} to \nba s.  
We show that a na\"ive blowup algebra that is generated in degree 1 gives an example of the geometry in Theorem~\ref{thm-main}.  Most results here are immediate consequences of results in 
\cite{NS1}.

We will need to work briefly with bimodule algebras (see  \cite{VdB1996}).  Our exposition is similar to that in  \cite{NS1}.

\begin{defn}\label{def-bimod}
Let $X$ be a finite type $\kk$-scheme.  An {\em $\sO_X$-bimodule}  is  a quasicoherent $\sO_{X \times X}$-module $\sh{F}$, such that for every coherent submodule $\sh{F}' \subseteq \sh{F}$,  the projection maps $p_1, p_2: \Supp \sh{F}' \to X$ are both finite morphisms.    The left and right $\sO_X$-module structures associated to an $\sO_X$-bimodule $\sh{F}$ are defined respectively as $(p_1)_* \sh{F}$ and $ (p_2)_* \sh{F}$.  
We make the notational convention that when we refer to an $\sO_X$-bimodule simply as an $\sO_X$-module, we are using the left-handed structure (for example, when we refer to the global sections or higher cohomology
 of an $\sO_X$-bimodule).   

There is a tensor product operation on the category of bimodules that has the expected properties; see \cite[Section~2]{VdB1996}.
\end{defn}

All the bimodules that we consider will be constructed from bimodules of the following form:
\begin{defn}\label{def-LR-structure}
Let $X$ be a projective scheme  and let $\sigma, \tau \in \Aut(X)$. Let $(\sigma, \tau)$ denote the map
\begin{displaymath}
X  \to X \times X \;\;\; \text{defined by} \;\;\;
x  \mapsto (\sigma(x), \tau(x)).
\end{displaymath}
If  $\sh{F}$ is a quasicoherent sheaf on $X$, we define the $\sO_X$-bimodule ${}_{\sigma} \sh{F}_{\tau}$ to be 
${}_{\sigma} \sh{F}_{\tau} = (\sigma, \tau)_* \sh{F}.$
If $\sigma = 1$ is the identity, we will often omit it; thus we write $\sh{F}_{\tau}$ for ${}_1 \sh{F}_{\tau}$ and write $\sh{F}$ for  the  $\sO_X$-bimodule ${}_1 \sh{F}_1 = \Delta_* \sh{F}$, where $\Delta: X \to X\times X$ is the diagonal. 
\end{defn}

\begin{defn}\label{def-BMA}
Let $X$ be a projective scheme.  
An {\em $\sO_X$-bimodule algebra}, or simply a {\em bimodule algebra}, $\sh{B}$ is an algebra object in the category of bimodules.  That is, there are a unit map $1: \sO_X \to \sh{B}$ and a product map $\mu: \sh{B} \otimes \sh{B} \to \sh{B}$ that have the usual properties.
\end{defn}

We follow \cite{KRS} and define
\begin{defn}\label{def-gradedBMA}
Let $X$ be a projective scheme  and let $\sigma \in \Aut(X)$.  
A bimodule algebra $\sh{B}$ is a {\em graded $(\sO_X, \sigma)$-bimodule algebra} if:

(1) There are coherent sheaves $\sh{B}_n$ on $X$ such that
$\sh{B} = \bigoplus_{n \in \ZZ} {}_1(\sh{B}_n)_{\sigma^n};$

(2)  $\sh{B}_0 = \sO_X $;

(3) the multiplication map $\mu$ is given by $\sO_X$-module maps
$\sh{B}_n \otimes \sh{B}_m^{\sigma^n} \to \sh{B}_{n+m}$, satisfying the obvious associativity conditions.  
\end{defn}

 For a specific example, let $X$ be a projective variety of dimension $\geq 2$.  Let $\sigma\in \Aut(X)$ and let $\sL$ be a $\sigma$-ample invertible sheaf on $X$.
As usual, we define $\sL_n := \sL \otimes \sL^\sigma \otimes \cdots \otimes \sL^{\sigma^{n-1}}$.
Let $P$ be a zero-dimensional subscheme of $X$.
 We define a bimodule algebra   $\sR = \sR(X, \sL, \sigma, P)$ as follows: 
we have $\sR = \bigoplus_{n \geq 0} \sR_n$, where
\[ \sR_n = \sI_P \sI_P^{\sigma} \dots \sI_P^{\sigma^{n-1}} \sL_n.\]
Then $\sR$ has a natural multiplicative structure induced by the multiplication maps
\[ \sR_n \otimes \sR_m^{\sigma^n} \to \sR_{n+m}. \]
Note that $ R(X, \sL, \sigma, P) = H^0(X, \sR)$.  

There is a natural way to define a (right) module over the bimodule algebra $\sR$.  We will only need this in the following specific case:
\begin{defn}\label{def:BMA module}
 Let $\sR := \sR(X, \sL, \sigma, P)$.  
The graded quasi-coherent sheaf $\sM := \bigoplus \sM_n$ on $X$ is a {\em graded right $\sR$-module} if there are maps $\sM_n \otimes \sR_m^{\sigma^n} \to \sM_{n+m}$ for all $n, m \in \NN$ that satisfy the appropriate associativity condition.  

The bimodule algebra $\sR$ is {\em (right) noetherian} if  (graded) right sub-$\sR$-modules of $\sR$ satisfy ACC.  

\end{defn}
This is equivalent to the definition in \cite{VdB1996}, as in the discussion after \cite[Definition~2.5]{KRS}.
If $\sI$ is an ideal sheaf on $X$, then $\sI\sR$ is a graded right $\sR$-submodule of $\sR$.

The algebra $R(X, \sL, \sigma, P)$ and the bimodule algebra $\sR(X, \sL, \sigma, P)$ may be defined for any 0-dimensional $P$, and in the sequel we will do so.
In  this section, however, we will  assume that 
 $P$ is supported on points with critically dense orbits.  Then  $\sR(X, \sL, \sigma, P)$ is right noetherian and $R(X, \sL, \sigma, P)$ is noetherian \cite[Proposition~2.12, Theorem~3.1]{RS-0}. 
As in \cite{NS1}, for any $\ell \in \NN$ we may define an fpqc-algebraic space ${}_{\ell} Z_{\infty}$, parameterising $\ell$-shifted point modules over $\sR$.
 A $\operatorname{Spec}(A)$-point of ${}_{\ell}Z_\infty$ is (up to isomorphism) a factor module $\sh{M}$ of $ \sh{R}_{\geq \ell} \otimes_{\kk}  A$ which is isomorphic as a graded $\sO_{X_A}$-module 
to a direct sum $\sP_\ell \oplus \sP_{\ell+1} \oplus \dots$. 
Here each $\sP_i$ is a coherent $\sO_{X_A}$-module such that its support  is finite over $\operatorname{Spec}(A)$ and is a rank one projective $A$-module.  
 Note that we have $\sP_m \otimes_X (\sR_1)^{\sigma^m} \twoheadrightarrow \sP_{m+1}$.  It follows that the support of $\sP_m$ on $X_A$ does not depend on $m$.
  
  The support of $\sP_m$ is by the above a section of the projection $\pi_2: X_A \to \operatorname{Spec}(A)$, or equivalently a $\operatorname{Spec}(A)$-point of $X$.
This gives a natural morphism
\begin{align*}
 r: {}_{\ell}Z_{\infty} & \to X \\
\sM & \mapsto \Supp \sM.
\end{align*}

As with $R$-modules, we have ${}_{\ell} Z_{\infty} = \varprojlim {}_\ell Z_n$, where ${}_\ell Z_n$ is a projective scheme parameterising appropriate truncations of $\ell$-shifted $\sR$-point modules.   
By \cite[Theorem~5.11]{NS1}, ${}_\ell Z_\infty$ is noetherian as an fpqc-algebraic space.  

Let $K = \kk(X)$ be the fraction field of  $X$, so we can view $\sigma \in \Aut(K)$.  Let $X_\infty \subseteq Z_\infty$ be the component containing the {\em generic shifted point module}
$K[t; \sigma]_{\geq \ell}$; we have $X_\infty = \varprojlim X_n$, where $X_n \subseteq {}_\ell Z_n$ is defined in the obvious way.

We first study the indeterminacy locus of $r^{-1}$, where we say that $r^{-1}$ is {\em defined at} $x \in X$ if in the fiber square
\begin{equation}\begin{gathered}\label{indeterminacy-diagram} \xymatrix{ ( \Spec \sO_{X,x})\times_X {}_{\ell} Z_{\infty} \ar[r] \ar[d]^{\mathsf{left}} & {}_{\ell} Z_{\infty} \ar[d]^{r} \\
 \Spec \sO_{X,x} \ar[r] & X}
\end{gathered}
\end{equation}
the left-hand vertical map $\mathsf{left}$ is an isomorphism.  
We immediately have:
\begin{lemma}\label{lem:indeterminacy of r}
 Let $X, \sL, \sigma, P$ as above, and let $\ell \in \NN$.  Let $r: {}_{\ell}Z_\infty \to X$ be the morphism defined above, and let $r' := r|_{X_{\infty}}$.  Then the indeterminacy locus of $r^{-1}$ and of $(r')^{-1}$ is precisely 
\[ \Omega:= \bigcup \{ \sigma^{-n}(\Supp P) \st n \geq 0 \}.\]
\end{lemma}
\begin{proof}
Let $x \in X \ssm \Omega$.  Consider the morphism
$i_x: \operatorname{Spec}(\mathcal{O}_{X,x})\times X \rightarrow X\times X$.  
Because all $\sR_n$ are invertible at $x$, the pullback $i^*_x\sR_{\geq \ell}$ is a shifted point module over $\operatorname{Spec}(\mathcal{O}_{X,x})$, hence defines a section of the morphism $\mathsf{left}$ of Diagram \eqref{indeterminacy-diagram}.
It is clear that
$i^*_x\sR_{\geq \ell}$ is the unique shifted point module over $\operatorname{Spec}(\mathcal{O}_{X,x})$; hence the scheme-theoretic fiber of $x$ is trivial.
Thus the indeterminacy locus of $r^{-1}$ is contained in $\Omega$.

For $n \in \NN$, let $\sI_n := \sI_P \cdots \sI^{\sigma^{n-1}}_P$.  
The first paragraph shows that $r^{-1}$ is defined at the generic point of $X$.   By \cite[Theorem~5.11]{NS1},  
\[ X_\infty \cong \varprojlim {\rm Bl}_{\sI_n} X,\]
where ${\rm Bl}_{\sI_n} X$ denotes the blowup of $X$ at $\sI_n$. 
(In fact, $X_n = {\rm Bl}_{\sI_n} X$.)
The indeterminacy locus of $(r')^{-1}: X \dra X_\infty$ is  $ \Omega$, and so 
 the indeterminacy locus of $r^{-1}$ contains and is thus equal to  $\Omega$.
\end{proof}

Let $\ell \in \NN$.  As in \cite{NS1}, for $n \geq \ell $ let ${}_{\ell} Y_n$ be the projective scheme parameterising factors of $R_{\geq \ell}$ with Hilbert series $t^{\ell} + t^{\ell +1} + \cdots + t^n$.  Let  ${}_{\ell} Y_\infty = \varprojlim {}_{\ell} Y_n$ be the fpqc-algebraic space which parameterises {\em $\ell$-shifted point modules of $R$}:  that is, factors of $R_{\geq \ell}$ with Hilbert series $t^{\ell}/(1-t)$.  
For $\ell \gg 0$, the spaces ${}_{\ell} Z_\infty$ and ${}_{\ell}Y_\infty$ are isomorphic, under a mild technical hypothesis. 

\begin{theorem} \label{thm:NS1}
(cf. \cite[Theorem~6.8]{NS1})
Let $X, \sL, \sigma, P$ be as above, and let $R := R(X, \sL, \sigma,P)$. 
\begin{enumerate}
 \item
Suppose that $R_n =(R_1)^n$ for $n \gg 0$.  Then there is $\ell_0 \in \NN$ so that for  $n \geq \ell \geq \ell_0$, the global sections functor induces  isomorphisms
\[ s_n:  {}_{\ell} Z_n \to {}_{\ell} Y_n \quad \quad s:  {}_{\ell} Z_\infty \to {}_{\ell} Y_\infty.\]
The inverse morphism $s^{-1}:{}_{\ell} Y_\infty \to {}_{\ell} Z_\infty$ acts as follows on $\kk$-points.  Let $y \in {}_{\ell} Y_\infty$ be represented by a graded module $S_{\geq \ell}/J$.  Then $s^{-1}(y)$ is represented by the factor $\sR_{\geq \ell}/\sJ$, where  $\sJ_n$ is the subsheaf of $\sR_n$ generated by the sections in $J_n$. The definition of $s_n^{-1}$ is similar. 
\item
If we replace $R$ by a sufficiently large Veronese, or alternatively replace $\sL$ by a sufficiently ample $\sL'$, then $R$ itself is generated in degree 1 and we may take $\ell_0 =0$ above.  
\end{enumerate}
\end{theorem}
\begin{proof}
$(1)$.  
 If $R$ is generated in degree 1, this is  \cite[Theorem~6.8]{NS1}.  The   statement  here is an immediate generalisation.

$(2)$.  By \cite[Proposition~3.18]{RS-0}, the Veronese $R^{(p)}$ is generated in degree 1 for all $p \gg 0$.  We then apply $(1)$ to $R^{(p)}$; if we replace
 $R^{(p)}$ by $R^{(p\ell_0)}$ then clearly we do not need to shift the point modules.  

  By \cite[Proposition~3.20]{RS-0}, if $\sM$ is ample and globally generated then some $R(X,  \sM^{\otimes m}, \sigma, P)$ is generated in degree 1.  That we may also take $\ell_0 = 0$ (possibly by changing the invertible sheaf again) is the comment after Definition~6.3 of \cite{NS1}.  
\end{proof}

In the next result, we use the isomorphism $s$
of Theorem~\ref{thm:NS1} to study the unshifted point schemes of a \nba.
We number the parts of this result to be consistent with Theorem~\ref{thm-main}.

\begin{proposition}\label{prop:existence of q}
 Let $X, \sL, \sigma, P$ be as above; in particular, assume that all points in $P$ have critically dense $\sigma$-orbits. 
 Let $R:= R(X, \sL, \sigma, P)$.  Suppose that $R_n = (R_1)^n$ for $n \gg 0$.

  Let $S$ be a graded subalgebra of $R$ so that $S$ is generated in degree 1 and  $S_n = R_n$ for $n \gg 0$.  
Let $Y_\infty$ be the point space for $S$, and let $\Psi:  Y_\infty \to Y_\infty$ be the shift morphism $M \mapsto M[1]_{\geq 0}$.

Then there is a morphism $q: Y_\infty \to X$ that satisfies the following properties:
\begin{itemize}
 \item[(i)] the morphism $q$ factors through $\pi$,  and $X$ corepresents $\H$ (in the category of schemes of finite type) via the factor morphism
\[ \xymatrix@C=5pt{
& F \cong  Y_{\infty} \ar[ld]_{\pi} \ar[rd]^{q} \\
\H \ar[rr] && X; }\]
\item[(iii)] the map $\H \to X$ is bijective on $\kk$-points;
\item[(iv)] the indeterminacy locus of $q^{-1}$ is $\bigcup \{ \sigma^{-n}(\Supp P) \st n \geq 0 \}$, and in particular is countable.
\end{itemize}
\noindent Furthermore,
\begin{itemize}
\item[(*)] $Y_\infty$ is a noetherian fpqc-algebraic space; and
 \item[(\dag)]
 $q \Psi = \sigma q$.
\end{itemize}
\end{proposition}
\begin{proof}
Let $\ell_0 \in \NN$ be given by Theorem~\ref{thm:NS1}(1), and let $\ell \geq \ell_0$ be such that $S_{\geq \ell } = R_{\geq \ell}$.  Then ${}_{\ell} Y_\infty$ also parameterises $\ell$-shifted point modules over $R$.

As in the proof of \cite[Theorem~6.8]{NS1}, 
there is a natural ``tail'' morphism
\[ T:  Y_\infty \to {}_{\ell} Y_\infty \]
given by restricting a surjection $f:  S \to M$ to $f_{\geq \ell}:  S_{\geq \ell} \to M_{\geq \ell}$. 
Let $q: Y_\infty \to X$ be the composition
\[ \xymatrix{
 Y_\infty \ar[r]^{T} \ar[rrd]_{q} & {}_{\ell} Y_{\infty} \ar[r]^{s^{-1}} & {}_{\ell} Z_{\infty} \ar[d]^{r} \\
& & X. }
\]

We verify that the conclusions of the proposition hold for $q$.    We note that for $n \gg 0$,  by Corollary~\ref{cor-stack=proscheme}  the morphisms $q$ and $r$ factor  to give $q_n:   Y'_n \to X$ and $r_n:  {}_{\ell} Z'_n \to X$  (here ${}_\ell Z'_n$ is the scheme-theoretic image of the natural map ${}_\ell Z_\infty \to {}_\ell Z_n$). 

(\dag).  
 If $\sM$ is a coherent graded right $\sR$-module, it follows as in \cite[Lemma~5.5]{KRS} that $\sM[1]_n \cong \sM_{n+1}^{\sigma^{-1}}$.  If $\sM$ is a (truncated) point module, therefore, we have 
$\Supp(\sM[1]) = \sigma(\Supp( \sM))$.  It follows that $q \Psi = \sigma q$.

(i).  
 Let $q'_n := \sigma^{-n} q$.
  Then $q'_{n+m} \Psi^m = \sigma^{-m} q'_n \Psi^m = q'_n$ so by Lemma~\ref{lem:Hdefn} there is an induced morphism $H \stackrel{t}{\to} X$ so that 
$\xymatrix{Y_{\infty} \ar[r]^{\pi} \ar[rd]_{q} & \H \ar[d]^{t} \\ & X}$ commutes.  Since $q\Psi= \sigma q$, we have $t \Psi = \sigma t$ as maps from $H \to X$.  

Now let $f: \H \to Y$ be a morphism  from $\H$ to  a scheme $Y$ of finite type.  
Recall the maps $\pi_n = \Psi^{-n} \pi:  Y_{\infty} \to H$ from the proof of Lemma~\ref{lem:Hdefn}; we have $\pi_n \Psi^n = \pi$ for all $n \in \NN$.   
For any $n \in \NN$, the map  $f \pi_n:  Y_{\infty} \to \H \to Y$ clearly factors through $G = Y_{\infty}/\sim$.
 By \cite[Proposition~7.3]{NS1},  $X$ corepresents $G$ in the category of schemes via a morphism $G\rightarrow X$, and it is easy to see that this morphism factors through $G \to H \stackrel{t}{\to} X$.
So there is a unique $\alpha_n:  X \to Y$ so that
\[ \xymatrix{
& Y_{\infty} \ar[ld]_{q} \ar[d]_{\pi} \ar[rd]^{\pi_n} & \\
X \ar[rd]_{\alpha_n}& H \ar[l]^{t} \ar[r]_{\Psi^{-n}} & H \ar[ld]^{f} \\
& Y & 
}
\]
commutes. 
Let $\alpha = \alpha_0$.
 Since $\alpha q = f \pi = f \pi_{m} \Psi^m = \alpha_{m} q \Psi^m = \alpha_{m} \sigma^m q$, by uniqueness of $\alpha$ we have $\alpha = \alpha_{m} \sigma^m$ for all $m \in \NN$.    

We have  $f \pi_n = \alpha_n q = \alpha \sigma^{-n} q = \alpha \sigma^{-n} t \pi = \alpha t \Psi^{-n} \pi = \alpha t \pi_n$ for all $n$.  By Lemma~\ref{lem:Hdefn} and the universal property of colimits, $f = \alpha t$, as we needed to prove. 

(iii). That $\H \to X$ is bijective on $\kk$-points follows from \cite[Theorem~1.2(4)(5)]{RS-0}.

(iv). By Lemma~\ref{lem:indeterminacy of r}, the indeterminacy locus of $r^{-1}$ is 
\[ \Omega := \bigcup \{ \sigma^{-n} (\Supp P) \st n \geq 0\}.\]
We claim that this is also the indeterminacy locus of $q^{-1}$.  

We first show  that $\im(q) \supseteq X \ssm \Omega$.
 To see this, let $\sR := \sR(X, \sL, \sigma, P)$.  
Let $x \in X \ssm \Omega$ be a $\kk$-point.  Define 
\[J_x := \bigoplus_{n\geq 0} S_n \cap H^0(X, \sI_x \sR_n).\]
The codimension of $(J_x)_n$ in $S_n$ is 0 or 1  for all $n$.  However, it follows from the fact that $R_n = (S_1)^n$ for $n \gg 0$ that the sections in $S_n$ generate $\sR_n$ for all $n \in \NN$.
Thus we cannot have $S_n = (J_x)_n$, and $M_x := S/J_x$ is a point module.
  From the construction of $s^{-1}$ in  Theorem~\ref{thm:NS1}, we see that $q(M_x) = x$.  
Recall that $q$ factors through $q_n:  Y'_n \to X$ for all $n \gg 0$.  Thus $\im q_n \supseteq X \ssm \Omega$ and, as $X$ is a variety, $q_n$ is scheme-theoretically surjective.

Let $Y^{(\ell)}_{\infty}$ be the point space   for the Veronese $R^{(\ell)} = S^{(\ell)}$.  
To give an $\ell$-shifted point module $M$ over $S$, it is equivalent to give a right ideal $J$ of $S$ so that $J_{< \ell} = 0$ and $J_{ \geq \ell}$ is the kernel of the surjection $S_{\geq \ell} \to M$.  
Let $J^{(\ell)} := \bigoplus_n J_{n\ell}$.  Then there is  a natural morphism $v: {}_{\ell} Y_\infty \to Y^{(\ell)}$, where $v(M) =  S^{(\ell)}/J^{(\ell)}$.  

We may clearly truncate $v$ and $T$ to obtain (compatible) morphisms $T_n:  Y_n \to {}_{\ell} Y_n$ and $v_{ n}:  {}_{\ell} Y_{\ell n} \to Y^{(\ell)}_n$, where $Y^{(\ell)}_n$ parameterises length $n$ truncated point modules over $R^{(\ell)}$.    
The composite $v_n T_{\ell n}$ sends a module $M$ of length $\ell n+1$ to $M^{(\ell)}$.  By \cite[Lemma~3.7]{RS-0}, $v_n T_{\ell n}$ is a closed immersion.  
Thus $T_{\ell n}$ is a closed immersion and so is $s_{\ell n}^{-1} T_{\ell n}:  Y_{\ell n} \to {}_{\ell} Z_{\ell n}$.   A diagram chase gives  that $s_{\ell n}^{-1} T_{\ell n}$ induces a closed immersion $Y'_{\ell n} \to {}_\ell  Z'_{\ell n}$.

If $x \in X \ssm \Omega$, let $C = \Spec \sO_{X,x}$.  
There is a commutative diagram: 
\[ \xymatrix{ C \times_X Y'_{\ell n} \ar[rr]^{C \times s^{-1}_{\ell n}T_{\ell n}} \ar[rrd]_{C \times q_{\ell n}} && C \times_X {}_\ell Z'_{\ell n} \ar[d]^{C \times r_{\ell n}} \\
 & &C.  }
\]
For $n \gg 0$, we have that $C \times r_{\ell n}$ is an isomorphism.  
As $C \times s^{-1}_{\ell n} T_{\ell n}$ is a closed immersion, so is $C \times q_{\ell n}$. 
As scheme-theoretic images commute with flat base change, $C \times q_{\ell n}$ is scheme-theoretically surjective.
Thus $C\times q_{\ell n}$ is an isomorphism for all $n \gg 0$.  It follows that $C \times q$ is an isomorphism.

Thus the indeterminacy locus of $q^{-1}$ is contained in $\Omega$.
On the other hand, it is clear from the above  that $s^{-1}T(Y_\infty)$ contains $X_\infty$, since each $s^{-1}_{\ell n} T_{\ell n} (Y_\infty)$ contains $X_{\ell n}$.  Thus the indeterminacy locus of $q^{-1}$ contains the indeterminacy locus of $(r|_{X_\infty})^{-1}$, which by Lemma~\ref{lem:indeterminacy of r} is $\Omega$.

(*).  The proof of \cite[Theorem~5.11]{NS1} constructs a positive integer $N$ and for each (sufficiently large) $n$ and $i \in \{1 \dots N\}$ open affine schemes $U_{n,i} \subset {}_{\ell}Z_n$,   together with  a surjective map $U_n = \coprod_i U_{n,i} \to {}_{\ell} Z_n$.  Further,  $U_\infty = \varprojlim U_n$ is a noetherian affine scheme and the induced surjective map $U_\infty \to {}_{\ell} Z_\infty$ is fpqc.    

For ease of notation, we denote ${}_{\ell} Z_n$ by $Z_n$ going forward.  
Let $V_{\ell n}$ be the fiber product $U_{\ell n} \times_{Z_n} Y_{\ell n}$, and let $V_\infty = \varprojlim V_{\ell n}$.  
As $T_{\ell n}: {}Y_{\ell n} \to Z_{\ell n}$ is a closed immersion, so is $V_{\ell n} \to U_{\ell n}$.  
It follows that $V_{\infty}$ is a closed subscheme of  $U_\infty$ and so is a noetherian affine scheme.  
Further, there is an induced map $V_\infty \to Y_\infty$ that is easily seen to be fpqc, representable, surjective, and formally \'etale; that the diagonal $V_\infty \times_{Y_\infty} V_\infty$ is representable, surjective, and quasicompact follows as in \cite[Theorem~5.11]{NS1}.  
Thus $Y_\infty$ is a noetherian fpqc-algebraic space. 
\end{proof}

\begin{corollary}\label{cor:hypotheses apply to NBAs}
Let $X, \sL, \sigma, P$ be  as above and let $R:= R(X, \sL, \sigma, P)$.  
Let $Y_\infty$ be the point space of $R$.  If $R$ is generated in degree 1, there is a morphism $p: Y_\infty \to X$ satisfying the hypotheses of Theorem~\ref{thm-main}.  
\end{corollary}
\begin{proof}
Hypotheses (i), (iii), (iv) follow directly from Proposition~\ref{prop:existence of q}, with $p = q$.   Hypothesis (ii) holds by assumption, since the points in $P$ have critically dense, and so dense, orbits.  Any point with a dense orbit is contained in the nonsingular locus of  $X$. 
\end{proof}

\section{Canonical birationally commutative factors and point parameter rings}\label{PPR}
In this section we construct a factor  of $R$
universal for maps from $R$ to birationally commutative
algebras.  Our results hold for any  connected graded algebra generated in degree 1 over an algebraically closed  field $\kk$, and we work in that generality.

We begin with some more  notation for point schemes.  
\begin{notation}\label{not1prime}
Let $R$ be a connected graded $\kk$-algebra generated in degree 1, and adopt Notation~\ref{not1}.  
For $N \in \NN$, let 
\[ Y_N^e := \bigcap_{i\in \NN} \phi^i \psi^i( Y_{2i+N}).\]
We refer to $Y_N^e$ as the {\em essential part} of $Y_N$.  Let $Y_\infty^e := \varprojlim_{\phi} Y_N^e $.   Let $Y'_n := \bigcap_i \phi^i(Y_{n+i})$.  By Lemma~\ref{lem:Yprime}, $Y'_n = \Phi_n(Y_\infty)$.  

Let $\PP:= \PP(R_1^*)$.  Then $Y_n$ naturally embeds in $\PP^{\times n}$ for any $n \in \NN$ (cf. \cite[Proposition~3.9]{ATV1990}), and there is a commutative diagram
\beq \label{fund}
\xymatrix{
Y_n \ar@<-.5ex>[d]_{\phi_n} \ar@<.5ex>[d]^{\psi_n} \ar[r] & \PP^{\times n}
\ar@<-.5ex>[d]_{\alpha_n} \ar@<.5ex>[d]^{\beta_n} \\
Y_{n-1} \ar[r]	& \PP^{\times (n-1)}.}\eeq
Here  
$\alpha_n(a_1, \ldots, a_n) = (a_1, \ldots, a_{n-1})$, and 
$\beta_n(a_1, \ldots, a_n) = (a_2, \ldots, a_n)$.
Let $\sM_n := \sO(1, \ldots, 1)|_{Y_n}$.  Let $\sM'_n := \sM_n|_{Y'_n}$, and let $\sM^e_n := \sM_n |_{Y_n^e}$.  
\end{notation}

We will want to be able to  restrict to the essential part $Y^e_\infty$.  To do this, we use:  

\begin{lemma}\label{lem-essential}
Let $R$ be any  connected graded $\kk$-algebra generated in degree 1, and adopt Notation~\ref{not1prime}.  
Then
\[ Y_N^e = \bigcap_{i,j\in \NN} \phi^i \psi^j( Y_{i+j+N}) = \bigcap_j \psi^j(Y'_{N+j})\]
for any $N \in \NN$.  
Further, 
\beq \label{chainapp} \psi(Y^e_{N+1}) = \phi(Y^e_{N+1}) = Y^e_N.\eeq
Thus, $\Phi_N (Y_\infty^e) = Y_N^e$, and
$Y_\infty^e = \bigcap_n \Psi^n(Y_\infty)$.  
In particular, 
$\Psi (Y_\infty^e) = Y_\infty^e$.  
\end{lemma}
\begin{proof}
Fix $N \in \NN$.  To prove the first equality, it suffices to show that $\{ \phi^k \psi^k
(Y_{2k+N}) \}_k$ is
coinitial 
in $\{\phi^i \psi^j( Y_{i+j+N})\}_{i,j}$, which is clear. The second equality is similar.
The equations \eqref{chainapp} follow from Lemma~\ref{lem:chain}, and by Lemma~\ref{lem:Yprime} we have  $\Phi_N (Y_\infty^e) = Y_N^e$.  

By definition, 
\[\Psi(Y_\infty^e)= \varprojlim \Phi_n \Psi(Y_\infty^e) = \varprojlim \psi \Phi_n(Y_\infty^e) = 
\varprojlim \psi(Y_n^e).\]
  We have seen this is $\varprojlim Y_{n-1}^e = Y_\infty^e$.
It follows that $Y_\infty^e \subseteq \bigcap_n \Psi^n(Y_\infty)$. 
On the other hand, 
\[\Phi_m(\bigcap_n \Psi^n(Y_\infty)) \subseteq \bigcap_n(\psi^n(Y'_{n+m}))=Y_m^e.\]
  Thus $\bigcap_n \Psi^n(Y_\infty) \subseteq \varprojlim Y_n^e$ and the two are equal.  
\end{proof}

If $R$ is strongly noetherian and generated in degree 1, then  the point schemes $Y_n$ stabilize for $n \gg 0$, and it follows that $Y_n = Y'_n = Y^e_n$ for $n \gg 0$. 
The need to distinguish carefully among the various point spaces is an unpleasant feature of the non-strongly noetherian case. 

We follow \cite{RZ2008}.  Adopt Notation~\ref{not1prime}.  Note that $Y_n$ carries a {\em universal (truncated) point module} $\sB_{\leq n} \cong (\phi^n)^* \sM_0   \oplus( \phi^{n-1})^* \sM_{1} \oplus  \cdots \oplus \sM_n$.  If $M$ is the truncated $R_A$-point module represented by 
$\delta_n:  \Spec A \to Y_n,$
then we have
\[ M \cong (\delta_n)^* \sB_{\leq n} \cong (\delta_0)^* \sM_0   \oplus( \delta_1)^* \sM_{1} \oplus  \cdots \oplus (\delta_n)^* \sM_n.\]
A  point module $M$ is determined by its truncations, so given an $R_A$-point module $M$ represented by $\delta:  \Spec A \to Y_\infty$, we have  $M \cong \bigoplus_{n\geq 0} \delta_n^* \sM_n$. 
(We avoid the technicalities necessary to define $M$ as the pullback of a universal point module defined on $Y_\infty$.)

As in \cite{RZ2008}, it is clear from the diagram \eqref{fund} that 
\beq\label{mult} \sM_{n+m} \cong (\phi^m)^* \sM_n \otimes (\psi^n)^* \sM_m, \eeq
and similarly for  $\sM'_n$ and $\sM_n^e$.  
Consider the natural  maps 
\[ \sM_n \to (\phi^m)_*(\phi^m)^* \sM_n \quad \mbox{ and } \quad
 \sM_m \to (\psi^n)_*(\psi^n)^* \sM_m.\]
We obtain maps
\begin{multline}\label{mult2}
H^0(Y_n, \sM_n) \otimes H^0(Y_m, \sM_m) \to H^0(Y_n, (\phi^m)_* (\phi^m)^* \sM_n ) \otimes H^0(Y_m, (\psi^n)_*(\psi^n)^* \sM_m)\\
= H^0(Y_{n+m}, (\phi^m)^* \sM_n) \otimes H^0(Y_{n+m}, (\psi^n)^* \sM_m) \to H^0(Y_{n+m}, \sM_{n+m}).
\end{multline}
The composition gives a multiplicative structure on 
\[ B := \bigoplus_{n \geq 0} H^0(Y_n, \sM_n)\]
and also on 
\[ B' := \bigoplus_{n \geq 0} H^0(Y'_n, \sM'_n) \quad\quad \mbox{ and }  \quad\quad B^e := \bigoplus_{n \geq 0} H^0(Y_n^e, \sM_n^e).\]
By an elementary calculation, these multiplications are associative.  There are obvious algebra homomorphisms $B \to B' \to B^e$.  

The ring $B$ is referred to as the {\em point parameter ring} of $R$, and its basic properties  are worked out in \cite{RZ2008}.   If $V$ is a vector space, we write $T(V)$ for the (free) tensor algebra on $V$.  
\begin{lemma}\label{lem-RZ}
{\rm (\cite[Lemma~4.1]{RZ2008})}
The natural map $T(R_1) \to B$ factors through $R$ to induce a homomorphism $\theta:  R \to B$.  An element $a$ is in $ (\ker \theta)_n$ if and only if  for all commutative $\kk$-algebras $A$  and  all truncated  $R_A$-point modules $M$ of length $n+1$, we have $M_0 a = 0$.
\end{lemma}

The map $\theta$ induces homomorphisms $\theta':  R \to B'$ and $\theta^e:  R\to B^e$.  We obtain immediately the following universal property of $\theta'$:
\begin{corollary}\label{cor-ann}
Let $R$ be a connected graded $\kk$-algebra generated in degree 1, and define $\theta':  R \to B'$ as above.  Then
\begin{multline*}
\ker \theta' = \bigcap \big\{ \Ann_R(M_0) \st \mbox{ $M$ is an $R_A$-point module for some commutative $\kk$-algebra $A$ } \big\} \\
= \bigcap \big\{ \Ann_R(M) \st \mbox{ $M$  is an $R_A$-point module for some commutative $\kk$-algebra $A$ } \big\}.
\end{multline*}
\end{corollary}

\begin{proof}
The first equality follows immediately from Lemma~\ref{lem-RZ} and the definition of $Y_n'$. 

For the second, let $M$ be an $R_A$-point module and let $a \in \ker \theta'$.  Then $M_k a = (\Psi^k M)_0 a = 0$ for any $k \in \NN$, so $Ma = 0$.  Thus
$\ker \theta' \subseteq \bigcap_M \Ann_R M \subseteq \bigcap_M \Ann_R(M_0) = \ker \theta'$.
\end{proof}

We are most interested in the ring $B^e$ and the map $\theta^e:  R \to B^e$.  Here we have an important universal property, which holds in very large
generality:  we do not even need $R$ to be noetherian. 

We define:  \begin{defn}\label{def-BC}
A {\em birationally commutative} algebra is a graded subalgebra of a skew polynomial extension $A[t; \tau]$, where $A$ is a commutative noetherian $\kk$-algebra (concentrated in degree 0) and $\tau \in \Aut(A)$. 
\end{defn}
Note that for our purposes here, we will require the commutative algebra in Definition \ref{def-BC} to be noetherian, although we caution that this is non-standard.

Then we have:

\begin{theorem}\label{thm-universal1}
Let $R$ be a connected graded $\kk$-algebra, generated in degree 1, and define $\theta^e:  R \to B^e$ as above.
Let $\alpha:  R \to \Delta$ be a homomorphism of graded $\kk$-algebras,
where $\Delta$ is birationally commutative.  Then $\alpha$ factors through $\theta^e$ up to finite dimension; that is, there
is some $n \in \NN$ so that $\ker \alpha \supseteq (\ker \theta^e)_{\geq n}$.
\end{theorem}
\begin{proof}
Without loss of generality, we may assume that $\Delta= A[t; \tau]$, where $A$ is a commutative noetherian $\kk$-algebra and $\tau \in \Aut_\kk
(A)$.  Let $S:= \alpha(R)$, and let $I:= A( S_1 t^{-1})
\subseteq A$.  Adopt Notation~\ref{not1}.

First suppose that $I = A$.  Then the natural map $R_A \to \Delta$
is
surjective, so $\Delta$ is an $R_A$-point module.  Since $Y_\infty$ parameterizes point modules, there is  a morphism
$\delta: \Spec A\to Y_{\infty}$, 
so that $A t^n = (\delta_n)^* \sM_n = \delta^* (\Phi_n)^* \sM_n$ for all $n \in \NN$.  

Since $\Delta[1]_{\geq 0} \cong \Delta^{\tau}$ (as an $ R_A$-module), the diagram
\[ \xymatrix{
\Spec A \ar[r]^{\delta_{n+m}} \ar[d]_{\tau^m} & Y'_{n+m} \ar[d]^{\psi^m} \\
\Spec A \ar[r]_{\delta_n} & Y'_{n} }
\]
commutes for all $n,m \in \NN$. 
Note that this shows the (scheme-theoretic) image of $\delta_n$ is contained in
$Y^e_n = \bigcap \psi^m( Y'_{n+m})$, as $\tau \in \Aut(A)$.    From the map
\[ \xymatrix{ B^e = \bigoplus H^0(Y_n^e, \sM_n^e) \ar[r] & \bigoplus H^0(\Spec A, (\delta_n)^*\sM^e_n) =
\Delta ,}
\]
we see that $\alpha$ factors through $\theta^e$. 

Now suppose that $I$ is a proper ideal of $A$.  Let $(0) = J_1 \cap \cdots \cap
J_s$ be a minimal primary decomposition in $A$, where $J_j$ is
$P_j$-primary. Reorder
the $P_j$ so that
$P_1, \ldots, P_r$ are the primes containing some $\tau^i(I)$, and
$P_{r+1},\ldots, P_s$ are the primes that do not contain any $\tau^i(I)$.  Let
$K := J_1 \cap \cdots \cap J_r$, and let $K':= J_{r+1} \cap \cdots \cap J_s$; we
may have $K = A$ or $K'=A$.  However, $K \cap K' = 0$, and $K$ and $K'$ are
$\tau$-invariant.  

Since $\tau$ permutes the finite set $P_1, \ldots, P_r$, there is some $N$ so that  
$I \tau(I) \cdots \tau^{N-1}(I) \subseteq K$.  Thus $S_n t^{-n} \cap K'
\subseteq K \cap K' = 0$ for $n \geq N$.  If $K'=A$ then this means that $S$ is
finite-dimensional, so the result holds.  Otherwise, let $\bar{\tau}$ be the
induced automorphism of $A/K'$, and let  $\pi:  A[t;\tau]\to
(A/K')[t;\bar{\tau}]$ be the natural map.   We see that $(\ker \alpha)_n =
(\ker \pi \alpha)_n$ for $n \geq N$, and we have reduced to the case that $K' = A$. 

It thus suffices to prove the proposition in the case that $I$ is  not
contained in
any associated prime of $A$.  Assume this is so.  Let
$A':= Q(A)$, the total ring of quotients of $A$. 
By prime avoidance
\cite[Lemma~3.3]{Eis}, $I$ is not contained in the union of the associated
primes of $A$, and so by 
  \cite[Theorem~3.1(b)]{Eis},   $I$ contains a regular element of $A$.  Thus  $IA'= A'$. Note that $\tau$ extends to
an automorphism $\tau'$ of $A'$.   There is an induced map $\alpha': R \to
A'[t; \tau']$; and as $I A' = A'$ this factors through $\theta^e$ by the first part
of the proof.  But $A \subseteq
A'$, so $\ker \alpha = \ker \alpha'$ and $\alpha$ also factors through $\theta^e$.
\end{proof}

We note that in the situation of the last paragraph of the proof,   $A'[t;
\tau']$ is torsion-free.  For, if $0 \neq a \in A'$, then $
A' a t^k \cdot R_n = a A' t^{n+k} \neq 0$, for any $n, k$.  Thus we have shown:
\begin{proposition}\label{prop-TRQ}
Let $R$ be a connected graded $\kk$-algebra generated in degree 1.  
Let $g:  R \to A[t; \tau]$ be a map of graded rings, where $A$ is commutative
noetherian and $\tau \in \Aut_\kk(A)$.  Let $I:=\ker g$.  Then there is a map
$\pi: A[t; \tau] \to A'[t; \tau']$ of commutative graded
$\kk$-algebras, where $A'$ is a noetherian total ring of quotients with an
automorphism $\tau'$ extending $\tau$, so that $\ker \pi g $ is equal to the
saturation of $I$.  Further, $\pi( g(R)) A' = A'[t; \tau']$.   \qed
\end{proposition}

Theorem~\ref{thm-universal1} shows that $\theta^e$ is universal for maps from $R$ to birationally commutative algebras.  We do not believe this has been observed before.  Motivated by this result, we (loosely) refer to the image $\theta^e(R)$ as the {\em canonical birationally commutative factor} of $R$; but note Example~\ref{eg-freealg} by way of caution.

\begin{example}\label{eg-freealg}
Let $V$ be a $d+1$-dimensional $\kk$-vector space and let $R := T(V)$ be the free  algebra on $V$.  Then $Y_n \cong (\PP^d)^{\times n}$, and $B = B'=B^e = \bigoplus_n H^0(Y_n, \sO(1, \ldots, 1)) \cong R$.  Thus $\theta^e(R)$ itself may not be birationally commutative.  We thank Chelsea Walton for pointing out this example.
 \end{example}

The differences between $B^e$, $B'$ and $B$, or alternatively between $\theta^e$, $\theta'$ and $\theta$, are fairly subtle. If $R$ is strongly noetherian, then the $Y_n$ stabilize, as mentioned.  
Then \cite[Theorem~1.1]{RZ2008} shows that 
$B$ (and therefore $B'$ and $B^e$) is equal in large degree to a  twisted homogeneous coordinate ring on the projective {\em scheme} $Y_\infty$, and that the map  $g: R \to B$ is surjective in large degree.  In fact, we have:
\begin{theorem}\label{thm-universal2}
 Let $R$ be a connected graded strongly noetherian algebra generated in degree 1, and let $g:  R \to B(X, \sL, \sigma)$ be the map constructed in Theorem~\ref{thm-RZ}.  
Then $g$ is universal for maps from $R$ to birationally commutative algebras, and $\ker g$, $\ker \theta^e$, $\ker \theta'$, and $\ker \theta$ are all equal in large degree.
Likewise, $B$, $B'$, $B^e$, and $B(X, \sL, \sigma)$ are all equal in large degree. 
\end{theorem}
\begin{proof}
 We certainly have $\ker \theta \subseteq \ker \theta' \subseteq \ker \theta^e$. 
 Since $B(X, \sL, \sigma)$ is birationally commutative, $(\ker \theta^e)_n \subseteq (\ker g)_n$ for $n \gg 0$.  
By \cite[Corollary~E4.12]{AZ2001}, for $n \gg 0$ we have that $\phi_n, \psi_n:  Y_n \to Y_{n-1}$ are isomorphisms.
It follows that $Y_n = Y_n^e$ for $n \gg 0$ and that $Y_\infty = Y_\infty^e$ is a projective scheme, isomorphic to $Y_n$ for $n \gg 0$.
The construction in the proof of \cite[Theorem~1.1]{RZ2008} gives $X = Y_\infty$, so $B(X, \sL, \sigma)$ and $B$ are equal in large degree.
Thus $\ker g$ and $\ker \theta$ are equal in large degree.
Since by \cite[Theorem~1.1]{RZ2008} $g$ is surjective in large degree, $B$, $B'$, $B^e$, and $B(X, \sL, \sigma)$ are equal in large degree.

The universal property of $\theta^e$ clearly also applies to $g$.
\end{proof}

In fact, Theorem~\ref{thm-universal2} holds more generally:  it is enough to assume that $R$ has subexponential growth and that $\phi_n$ and $\psi_n$ are isomorphisms for $n\gg 0$.  
See \cite[Theorem~4.4]{RZ2008}.

We suspect that if $R$ is noetherian, then $\ker \theta$ and $\ker \theta^e$ are equal in large degree; we will see later  that,  in the situation of Theorem~\ref{thm-main}, $\ker \theta'$ and $\ker \theta^e$ are in fact equal.  
 It is possible that $\theta^e$ is always  surjective (in large degree); we do not know of an example of a ring generated in degree 1 where this fails.  
Even under the hypotheses of Theorem~\ref{thm-main}, however, we have not proved these statements.

\section{Coarse moduli:  first properties}\label{COARSE}
We now begin to work towards the proof of Theorem~\ref{thm-main}.  
In this section, we assume the existence of a scheme $X$ that corepresents $\H$:  that is, $X$ is a {\em coarse moduli space} for $\H$.  
We construct an induced automorphism $\sigma$ of $X$ and study some of its properties.  
For simplicity, we assume in the next result that $X$ is a variety, although this is not strictly necessary.

 Let $R$ be a connected graded noetherian $\kk$-algebra generated in degree 1, and adopt Notation~\ref{not1}.

   We will say that $p:  Y_\infty \to X$ {\em is an isomorphism in codimension $d$} if $p^{-1}$ is defined at all points in $X$ of codimension $d$. 

\begin{proposition}\label{prop-II}
Let $R$ be a connected graded noetherian $\kk$-algebra generated in degree 1, and adopt Notation~\ref{not1prime}. 
Further suppose that the projective variety $X$ corepresents $\H$.  

Define  $p:  Y_\infty \to X$  through the commutative diagram
\[ \xymatrix@C=5pt{
& F \cong  Y_{\infty} \ar[ld]_{\pi} \ar[rd]^{p} \\
\H \ar[rr] && X. }\]
Then:  
\begin{enumerate}
\item the
endomorphism $\Psi$ of $Y_{\infty}$ induces an automorphism $\sigma$
of $X$;
\item  $p(  Y^e_\infty )= X$;
\item if  $p$ is an isomorphism in codimension 0, then there is a $\kk$-algebra homomorphism   $g:  R \to K[t;\sigma]$, where $K = \kk(X)$ is the  function field of  $X$.
\end{enumerate}
\end{proposition}
\begin{proof}
$(1)$
The shift map $\Psi$ is trivially an automorphism of $\H$.  
Since $X$ corepresents $\H$, the map $\H \stackrel{\Psi}{\to} \H \to X$ must factor through $X$, giving a unique $\sigma:X \to X$ so that the diagram
\beq\label{diag} \xymatrix{ \H \ar[r]^{\Psi} \ar[d] & \H \ar[d] \\ X\ar[r]_{\sigma} & X}\eeq
commutes. 
Likewise, there is a unique $\sigma':  X \to X$ so that 
\[ \xymatrix{ \H \ar[r]^{\Psi^{-1}} \ar[d] & \H \ar[d] \\ X\ar[r]_{\sigma'} & X}\]
commutes.  It is immediate that $\sigma'$ is the inverse of $\sigma$, and in particular that $\sigma$ is an automorphism.

$(2)$
By Proposition~\ref{prop-surj-factor}, for some $N\gg 0 $ there is a morphism $f_N: Y'_N \to X$ so that $p = f_N \Phi_N$.
For $n \geq N$ let $f_n:= f_N \phi^{n-N}:  Y'_n \to X$.
Then 
 $p \Psi^m = f_N \psi^m \Phi_{N+m}$ for all $m$.
Note that 
\beq\label{marchmont}f_n \psi = \sigma f_{n+1}.\eeq

By Lemma~\ref{lem-essential} we have $Y_N^e = \bigcap_m \psi^m (Y'_{N+m})$.  As $Y_N$ is noetherian, there is some $M$ so that $Y_N^e = \psi^M (Y'_{N+M})$.
As $X$ corepresents $\H$, certainly $p$ is scheme-theoretically surjective, and so $\sigma^M p$ is scheme-theoretically surjective.
By commutativity of \eqref{diag}, we have $X = \sigma^M p(Y_{\infty}) = p \Psi^M(Y_{\infty}) = f_N \psi^M(Y'_{N+M}) = f_N (Y_N^e) = p(Y_\infty^e)$.

$(3)$  Let  $\eta$ be the generic point  of $X$.  By assumption $p^{-1}$ is defined at  $\eta$; let $\zeta = p^{-1}(\eta)$.  
Let $C= \operatorname{Spec} \sO_{X,\eta}$.
Then there is a commutative diagram
\[ \xymatrix{ Y_\infty \times_X C \ar[r]^{\Phi_n \times  C} \ar[rd]_{p\times C}^{\cong} & Y'_n \times_X C \ar[d]^{f_n \times C} \\
& C}\]
for all $n \geq N$.
Since scheme-theoretic surjectivity is preserved under flat base extension, 
$\Phi_n \times C$ is scheme-theoretically surjective.
By Lemma ~\ref{lem-surj-equal}(\ref{scholium}), $f_n \times C$ is an isomorphism, so $(f_n)^{-1}$ is defined at $\eta$ and $(\Phi_n)^{-1}$ is defined at  $\Phi_n(\zeta) = f_n^{-1}(\eta)$.  

For $n\geq N$, there is a closed subscheme $X_n \subseteq Y^e_n$ so that $f_n: X_n \to X$ is birational.   For $0 \leq n < N$, let $X_n = \phi^{N-n}(X_N)$.
It is a consequence of \eqref{marchmont} that  $\psi(X_{n+1}) = X_{n}$ for all $n \in \NN$; for $n \geq N$, we have that $\psi_{n+1}:  X_{n+1} \to X_n$ is birational.
 
Let $\sK_n$ be the sheaf of total quotient rings of
$X_n$.  Then $H^0(X_n, \sK_n)$ is a total ring of quotients, and
$\phi^*, \psi^*: 
H^0(X_n, \sK_n) \to H^0(X_{n+1}, \sK_{n+1})$ are injective for $n \in \NN$ and are ring isomorphisms for $n
\geq N$.   Let $K:= \kk(\eta) = H^0(X_N, \sK_N)$.

 Let $i_N:  \Spec K \to X_N$ be the inclusion, and let $i_n := \phi^{N-n} i_N:  \Spec K \to X_n \subseteq Y_n$ for $n \in \NN$.  (Our choice of $N$ ensures this is always well-defined.)  
Let $\tau:=  i_N^{-1} \psi_{N+1} \phi^{-1}_{N+1} i_N: \Spec K \to \Spec K$, and let 
$\tau :  K \to K$ also denote the induced algebra automorphism.  We have  
\beq \label{star}
\psi_{n+1} i_{n+1} = i_n \tau
\eeq
for all $n \in \NN$.

For $n \in \NN$, let $\sL_n:= i_n^* \sM_n$.  
Let $S := \bigoplus_{n \geq 0} H^0(\Spec K, \sL_n)$.  This is
 a $\kk$-algebra, with multiplication induced from \eqref{mult}.  Each $S_n$ is a 1-dimensional  $K$-vector space. 
  From \eqref{star}, we see that 
$S \cong K[t;\tau]$.   There is a natural algebra map $ B \to S$, and composing with $\theta$ we obtain the desired map from $R \to S$.  
\end{proof}

In this generality, we cannot say much about the map $g$ or about $g(R)$---we do not know, for example, if $g(R)$ is the canonical birationally commutative factor of $R$ up to finite dimension.  We do note that if we assume that  $p^{-1}$ is defined in codimension 1 and $X$ is locally factorial at points in the indeterminacy locus of $p^{-1}$, then it is not hard to show (using an argument similar to that in Proposition~\ref{prop-N-L}) that $g(R)$ is contained in a twisted homogeneous coordinate ring $B(X, \sL, \sigma)$, where $\sL$ is an invertible sheaf on $X$.  We do not know, however, if $\sL$ must be $\sigma$-ample.  

We   note that Proposition~\ref{prop-II} applies to the situation in Theorem~\ref{thm-main} that we are most interested in.

\begin{lemma}\label{lem-birational}
Adopt Notation~\ref{not1}, and assume that the hypotheses of Theorem~\ref{thm-main} hold.  Then the morphism $p$ is birational.  More precisely, let $\eta$ be the generic point
of  $X$, with function field
$K = \kk(\eta)$.  Then $p^{-1}$ is defined  at $\eta$.  Thus conclusions $(1)$--$(3)$ of Proposition~\ref{prop-II} hold for $R$.
\end{lemma}
\begin{proof}
Adopt Notation~\ref{not1prime}.  
By Proposition~\ref{prop-surj-factor}, $p: Y_\infty \to X$ factors through $\Phi_{n_0}$ for some $n_0$.  Let $f_{n_0}: Y'_{n_0} \to X$ be the induced map.  For $n \geq n_0$, let 
$
f_n := f_{n_0} \phi^{n-n_0}:  Y'_n \to X.
$
Let $\Omega$ be the  indeterminacy locus of $p^{-1}$; by assumption, $\Omega$ consists of countably many
$\kk$-points.  
If $n \geq n_0$, then the indeterminacy locus of  $f_n^{-1}$ is a closed subset of $X$ that is contained in $\Omega$:  it is thus finite.    Thus $f_n^{-1}$ is defined at $\eta$; taking the limit, we see
that $p^{-1}$ is defined at $\eta$.  
\end{proof}

We now suppose that the coarse moduli space $X$ exists and in addition that $p: Y_\infty \to X$ is an isomorphism in codimension 1.  This allows us to construct a sequence of reflexive sheaves on $X$ whose properties we now study.  

\begin{notation}\label{not3}
Adopt Notation~\ref{not1prime}, and suppose that there is a projective variety $X$ that corepresents $\H$.  Suppose in addition that $p$ is an isomorphism in codimension 1. 

Using Proposition~\ref{prop-surj-factor}, let $n_0$ be such that $p$ factors through $\Phi_n$ for $n \geq n_0$, and let $f_n:  Y'_n \to X$ be the induced map.  
As in the proof of Proposition~\ref{prop-II}(3),  $(f_n)^{-1}$ is defined in codimension 1 for $n \geq n_0$.

For $n \geq n_0$, let $W_n$ be the indeterminacy locus of $(f_n)^{-1}:  X \dra Y'_n$.  By assumption, $W_n$ has codimension 2 in $X$.  Let $U_n:=X \smallsetminus W_n$ and let $i_n: U_n \to X$ be
the inclusion morphism. 

 If $n \geq n_0$, let $\sR'_n := (f_n)_* \sM'_n$ and let $\sN_n := (i_n)_*
(i_n)^* \sR'_n$.  
If $1 \leq n < n_0$, let $\sR'_n := (f_{n_0})_* (\phi^{n_0-n})^* \sM'_n$ 
and let
$\sN_n := (i_{n_0})_* (i_{n_0})^* \sR'_n$.  

Let $K := \kk(X)$.  Let $\sigma \in \Aut(X)$ and $g:  R \to K[t; \sigma]$ be given by Proposition~\ref{prop-II}.
Let  $\sL:=(\sN_{n_0+1}\otimes (\sN_{n_0}^\vee)^\sigma)^{\vee\vee}$.  
 For $n \geq 0$, let
\[ \sL_n := (\sL \otimes \sL^{\sigma} \otimes \cdots \otimes
\sL^{\sigma^{n-1}})^{\vee\vee}.\]
Let $\sR'_0 := \sN'_0 := \sO_X$.
\end{notation}

We collect some basic properties of the sheaves $\sN_n$, $\sL_n$, and $\sR'_n$.

\begin{proposition}\label{prop-N-L}
Adopt Notation~\ref{not3}.  In particular, suppose that there is a projective variety $X$ that corepresents $\H$ and  that $p$ is an isomorphism in codimension 1.
\begin{enumerate}
 \item The sheaves $\sN_n$ are reflexive for $n \geq 0$, and  $\sN_n = (\sR'_n)^{\vee\vee}$.
\item For $n \geq 0$ the  the natural map $\sR'_n \to
\sN_n$ is an isomorphism in codimension 1.
\item There are ``multiplication'' maps
\[  \sR'_n \otimes (\sR'_m)^{\sigma^n}\to  \sR'_{n+m}\]
for all $n \geq 0$ and $m \geq n_0$, satisfying the obvious associativity conditions.  In particular, $\sO_X \oplus \bigoplus_{n \geq n_0} \sR_n$ is a bimodule algebra. (See Section~\ref{NBEXAMPLE}.)
\item  For any $k \geq n_0$, we have $\sL\cong(\sN_{k+1}\otimes (\sN_{k}^\vee)^\sigma)^{\vee\vee}$.  
\item In fact, for any  $n \geq 0$, we have  $\sN_n \cong \sL_n$.
\end{enumerate}
\end{proposition}
\begin{proof}
$(1)$, $(2)$.  Let $n \geq n_0$ and let $i:= i_n$.  Away from the codimension 2
set $W_n$, the sheaf $\sR'_n $ is invertible, as $f_n$ is a
local
isomorphism. 
  The kernel and
cokernel of $\sR'_n \to \sN_n$ are supported on $W_n$.  This proves $(1)$ and
$(2)$ for $n \geq n_0$; the proof for $n < n_0$ is similar, using $W_{n_0}$.

$(3)$.  Let $m, n \geq n_0$.  The natural morphisms
\[ \sM'_n \to (\phi^m)_* (\phi^m)^* \sM'_n\]
and
\[ \sM'_m \to (\psi^n)_* (\psi^n)^* \sM'_m\]
induce maps
\[ \sR'_n \to (f_n)_*(\phi^m)_* (\phi^m)^* \sM'_n = (f_{n+m})_* (\phi^m)^*
\sM'_n\]
and
\[  (\sigma^n)^* \sR'_m  \to (\sigma^n)^* (f_m)_* (\psi^n)_* (\psi^n)^* \sM'_m =
(\sigma^n)^* (\sigma^n)_* (f_{m+n})_* (\psi^n)^* \sM'_m = (f_{m+n})_* (\psi^n)^*
\sM'_m.\]
If $0 \leq n< n_0$ and $m \geq n_0$, there is a map
\[ \sR'_n = (f_{n_0})_* (\phi^{n_0-n})^* \sM_n' \to (f_{n_0})_* (\phi^{m-n_0+n})_* (\phi^{m-n_0+n})^* (\phi^{n_0-n})^* \sM_n' = (f_{n+m})_* (\phi^m)^* \sM'_n.\]
For all $n \geq 0$ and $m \geq n_0$, the isomorphism 
\beq\label{mult-M}
(\phi^m)^*\sM'_n \otimes (\psi^n)^* \sM'_m \to \sM'_{n+m}
\eeq
observed in Section~\ref{PPR}
thus induces a multiplication map
\[
\sR'_n \otimes (\sR'_m)^{\sigma^n} \to
(f_{n+m})_* (\phi^m)^* \sM'_n \otimes (f_{n+m})_* (\psi^n)^* \sM'_m \to
(f_{n+m})_* \sM'_{n+m} = \sR'_{n+m}.
\]

$(4)$.  
If $k > n_0$, then from \eqref{mult-M} we have
\[
\phi^*(\sM'_k\otimes (\psi^* \sM'_{k-1})^{-1}) \cong \sO(1, 0^k)|_{Y'_{k+1}} \cong
\sM'_{k+1}\otimes (\psi^* \sM'_k)^{-1}.
\]
Working on the open set where $f_{k+1}$ is an isomorphism, we see that the
isomorphism class of $\sL$ does not depend on the integer $k \geq n_0$ used to
define it.

$(5)$.  The multiplication  maps induce morphisms $\sN_n \otimes \sN_m^{\sigma^n} \to
\sN_{n+m}$ 
for all $n \geq 0$, $m \geq n_0$.  Since $f_{n+m}$ is an isomorphism away from a codimension 2 locus in $X$ and the $\sN_k$ are reflexive,  the induced maps 
$(\sN_n \otimes \sN_m^{\sigma^n})^{\vee\vee}  = (i_{n+m})_*i_{n+m}^*(\sN_n \otimes \sN_m^{\sigma^n}) \to
\sN_{n+m}$ 
must in fact be isomorphisms.
Further, these maps are associative, since the corresponding property holds for
\eqref{mult-M}.

 Let $k\geq n_0$ and $n \geq 0$; then we have
\begin{multline*}
\sN_n \cong \big(\sN_{n+k} \otimes (\sN_{k}^{\vee})^{\sigma^n}\big)^{\vee\vee} \cong \\
\big(\sN_{n+k} \otimes (\sN_{n+k-1}^{\vee})^{\sigma} \otimes \sN_{n+k-1}^{\sigma}
\otimes (\sN_{n+k-2}^{\vee})^{\sigma^2} \otimes \cdots \otimes
\sN_{k+1}^{\sigma^{n-1}}\otimes(\sN_{k}^{\vee})^{\sigma^n}\big)^{\vee\vee} \cong \big(\sL \otimes
\sL^{\sigma} \otimes \cdots \otimes \sL^{\sigma^{n-1}}\big)^{\vee\vee} = \sL_n.
\end{multline*} 
\end{proof}

The bimodule $\bigoplus_{n \in \NN} (\sL_n)_{\sigma^n}$ is also a bimodule algebra.  
The global sections of a bimodule algebra have a natural algebra structure, and one can show that the natural map $R \to \bigoplus H^0(X, \sN_n) \cong \bigoplus H^0(X, \sL_n)$ is in fact an algebra homomorphism. However, it is easier to   work instead with the map $g:  R \to K[t; \sigma]$ defined in Notation~\ref{not3} and given by Proposition~\ref{prop-II}, and we do so.

\section{Points and  curves in  $Y_\infty$}\label{POINTS} 
To prove Theorem~\ref{thm-main}, we must understand the  geometry of  point spaces at a finer level, and in particular, study the structure of the countable subset of $X$ consisting of indeterminacy points of $p^{-1}$.  We next focus on curves.  
Note the next result holds for any connected graded noetherian $R$ generated in degree 1, without further assumptions on the structure of $Y_\infty$ or on the cardinality of $\kk$.

\begin{proposition}\label{prop-VII}
Let $\kk$ be an algebraically closed field and let $R$ be a connected graded noetherian $\kk$-algebra generated in degree 1.  Adopt Notation~\ref{not1}.  Let $C$ be an irreducible projective curve in $Y_{n}$; suppose that $\Phi^{-1}_n$ is
defined at the generic point of $C$.  Then
$C$ contains only finitely many  points of indeterminacy of $\Phi_{n}^{-1}$. 
\end{proposition}
\begin{proof}
Let $\Ct$ be the normalization of $C$.  Then the map $\Ct \to C \to  Y_{n}$ lifts via
the morphisms $\phi^m$ to maps $\Ct \to Y_{n+m}$ for all $m \geq 0$. 
These maps  stabilize for $m \gg 0$, by finiteness of the integral closure, to induce  morphisms
\[ \xymatrix{
\Ct \ar[r]_j & \hat{C} \ar[r]_i & Y_{\infty}
}\]
for some projective curve $\hat{C}$, 
where $j$ is birational and $i$ is a closed immersion.

Let $U = \Spec A$ be an open affine subset of $\hat{C}$.  Let $M$ be the $A$-point
module associated to 
\[ U \to \hat{C} \to Y_{\infty}.\]
 Since $A$ is an affine commutative $\kk$-algebra, the ring $R_A$ is noetherian.   Thus the right ideal $I$ given by 
\[ 0 \to I \to R_A \to M \to 0 \]
is finitely generated, in degrees $\leq N$.  If $x$ is any $\kk$-point of $U$,
the associated $\kk$-point module is $M_x:= M \otimes_A \kk_x$.  Since $M$ is
$A$-flat, the sequence
\[ 0 \to I\otimes_A \kk_x \to R \to M_x \to 0\]
remains exact.  In particular, the right ideal defining $M_x$ is generated in
degrees $\leq N$.   This precisely says that for $m \geq N$, the map $\Phi_m:
Y_{\infty} \to Y_m$ is a local isomorphism at $x$, and since $x \in U$ was arbitrary, at all points of $U$.  

Covering
$\hat{C}$ by finitely many open affines and increasing $N$ if necessary, we see  that $\Phi_N$ is a local isomorphism at all points of $\hat{C}$. 
Thus the only components of the indeterminacy locus of $\Phi_{n}^{-1}$ that
$C$ may possibly meet are those in the indeterminacy locus of $(\phi^{n-N})^{-1}: Y_N
\dra Y_{n}$.  This is a proper closed subscheme of $C$ and thus is finite.  
\end{proof}

In the next few results, we consider the following slightly weaker version of the hypotheses of Theorem~\ref{thm-main}:
\begin{hypotheses}\label{hyp-FOO}
 Let $\kk$ be an uncountable algebraically closed field, and let $R $ be a  noetherian connected graded $\kk$-algebra generated in degree 1. Adopt Notation~\ref{not1}.
Suppose the following:
\begin{itemize}
 \item[(i)]
there is a commutative
diagram
\[ \xymatrix@C=5pt{
& F \cong  Y_{\infty} \ar[ld]_{\pi} \ar[rd]^{p} \\
\H \ar[rr] && X }\]
where $X$ is a a projective scheme that corepresents $\H$ through the morphism $H \longrightarrow X$;
\item[(ii)]  $X$ is a variety of dimension $\geq 2$;
\item[(iii)] the map  $\H \to X$ is bijective on $\kk$-points; 
\item[(iv)]  the indeterminacy locus of $p^{-1}$ 
consists  (set-theoretically) of
countably many points.
\end{itemize}
\end{hypotheses}
In particular, we alert the reader that from here on, we will assume that $\kk$ is uncountable.   
\begin{corollary}\label{cor-curves}
 If  Hypotheses~\ref{hyp-FOO} hold, then any curve in $X$  contains only  finitely many points of
indeterminacy of $p^{-1}$.  \qed 
\end{corollary}

We next study preimages of $\kk$-points in $Y_\infty$.  We write a (not necessarily closed) point of $Y_\infty$ as $y = (y_n)$, with $y_n = \Phi_n(y) \in Y'_n$.

\begin{proposition}\label{prop-contract}
 Adopt Notation~\ref{not1}, and assume that  Hypotheses~\ref{hyp-FOO} hold.  
Let $x $ be a $\kk$-point of $X$, and let $y, z$ be (not necessarily closed)
points of  $p^{-1} (x) \subseteq Y_{\infty}$.   Then $\Psi^k (y) = \Psi^k (z)$ for  $k\gg 0$.  In particular, $\Psi^k (z)$ is a $\kk$-point of
$Y_{\infty}$ for $k \gg 0$.
\end{proposition}
\begin{proof}
If $y, z$ are $\kk$-points, this follows from the fact that $\H \to X$ is a
bijection on $\kk$-points.
It suffices, then, to prove the proposition in the case that $z$
is not a $\kk$-point and $y$ is a $\kk$-point; in fact,  by the first sentence of the proof 
it suffices to prove the proposition for $z$ and for {\em one} $\kk$-rational
point $y' \in p^{-1}(x)$.  

Let $z=(z_n)$ and let $Z_n := \bbar{\{z_n\}} \subseteq Y'_n$.  By Proposition~\ref{prop-RS}, there is
some $N\in \NN$ so that for $n \geq N$ the rational map $\phi^{-1}_{n+1}$ is defined and is a local isomorphism at $z_n$.  Thus the fiber $\phi^{-1}_{n+1} (z_n)$ is a singleton,  equal to 
$ z_{n+1}$.  That is, for $n \geq N$ the morphism $\phi_{n+1}:  Z_{n+1} \to Z_{n}$
is
birational and scheme-theoretically surjective, and $\kk(z_n) = \kk(z)$.
Let $Z_\infty:= \varprojlim_{\phi} Z_n$.  Let $y' = (y'_n)$ be a $\kk$-point of $Z_\infty \subseteq Y_\infty$.

For all $k \in \NN$, let
\[ Z^{(k)} = Z_N \cap \bigcap_{n \geq N+k} \phi^{n-N} ((\psi^k)^{-1} \psi^k (y'_n)).\]
This is a closed subscheme of $Z_N$, since $\phi$ is  proper and $y'_n$ is a closed point.  We have $Z^{(k)} \subseteq Z^{(k+1)}$.

If $y'' \in Z_\infty$ is a $\kk$-point, then $\Psi^k(y'') = \Psi^k(y')$ for some
$k$ by the first sentence of the proof.  
Thus, the countable union 
$\bigcup_k Z^{(k)}$ contains all $\kk$-points of $\Phi_N(Z_\infty) = Z_N$.  Since $\kk$ is uncountable, we must have some $Z^{(k)} = Z_N$
by the following well-known fact:
\begin{lemma}
Suppose $Z$ is an irreducible variety over a field $\kk$ of cardinality $\aleph$.  If $S$ is a set of cardinality $\beth$ strictly less than $\aleph$ and $V\subset Z$ is a union of proper closed subsets $V_i$, $i \in S$, then $Z\smallsetminus V$ contains a closed point.
\end{lemma}
We are grateful to Ravi Vakil for the following argument.
\begin{proof}
We may assume that  $Z$ is affine and that each $V_i$ is an irreducible hypersurface.   Suppose the statement is true for $Z=\mathbb{A}^n$, where $n=\operatorname{dim}(Z)$.  Then 
Noether normalization gives a finite surjective map $\phi: Z\rightarrow \mathbb{A}^n$ and the images of the $V_i$ under $\phi$ are irreducible
hypersurfaces; if $x\in \mathbb{A}^n\smallsetminus \bigcup_{i\in S} \phi(V_i)$ is closed, so is $\phi^{-1}(x)$.  So it suffices to establish the case 
$Z=\mathbb{A}^n$, which we do by induction; $n=1$ is clear.  For the inductive step, project $\mathbb{A}^n\rightarrow \mathbb{A}^1$; only 
$\beth$-many of the fibers can be among the $V_i$, so we can choose a fiber which is not among the $V_i$ and, by the inductive hypothesis, a closed point in it.  
\end{proof}

Let $n \geq N+k$.  As $\phi^{n-N}((\psi^k)^{-1} \psi^k(y'_n)) \ni z_N$, by choice of $N$ we have 
\[ z_n = (\phi^{n-N})^{-1}(z_N)  \in (\psi^k)^{-1} \psi^k(y'_n),\]
and $\psi^k(z_n) = \psi^k(y'_n)$.  
Thus $\Psi^k(z) = \Psi^k(y')$.
\end{proof}

\begin{corollary}\label{cor-1pt}
Adopt Notation~\ref{not1}, and assume that  Hypotheses~\ref{hyp-FOO} hold.  
Let $x$ be a $\sigma$-fixed $\kk$-point of $X$.  Then $p^{-1}(x) =  \{v\}$ is a single $\kk$-point, and $\Psi^{-1}$ is defined and is a local isomorphism at $v$.
\end{corollary}
\begin{proof}
Let $z \in p^{-1}(x)$.  As $p\Psi(z) = \sigma p(z) = x$, by Proposition~\ref{prop-contract} for some $k$,  $\Psi^k(z) = \Psi^{k+1}(z) =: v$ is a $\kk$-point of $Y_\infty$ and of $p^{-1}(x)$.  Let $v= (v_n)$. 
    As $\Psi(v) = v$, we have  $\psi_{n+1}(v_{n+1}) = \phi_{n+1}(v_{n+1}) = v_{n}$ for all $n \in \NN$.  By Proposition~\ref{prop-RS}, for $n \gg 0$ both $\phi_n^{-1}$ and $\psi_n^{-1}$ are defined and are local isomorphisms at $v_{n-1}$.   Since $\Psi = \varprojlim_{\phi_n} \psi_n$, it follows that $\Psi^{-1}$ is defined and is a local isomorphism at $v$.  
By Proposition~\ref{prop-contract}, $p^{-1}(x) = \bigcup_{k \in \NN} \Psi^{-k}(v)$.
Thus $p^{-1}(x) =  \{v\}$ is a single $\kk$-point.    
\end{proof}

We now prove a key result:  that the point space $Y_\infty$ differs from the coarse moduli space $X$ only at points of infinite order.

\begin{theorem}\label{thm-VIII}
Adopt Notation~\ref{not1}, and assume that  Hypotheses~\ref{hyp-FOO} hold. 
Then all points where $p^{-1}:  X \dra Y_\infty$ is undefined have  infinite $\sigma$-orbits.
\end{theorem}
\begin{proof}
We adopt Notation~\ref{not3} to obtain maps $f_n:  Y'_n \to X$ for all $n \geq n_0$,  
where the indeterminacy locus of $f_n^{-1}$ consists of isolated points.  

Suppose that $x \in X$ has a finite orbit.  By passing to a Veronese of $R$, and
thus replacing $\sigma$ by $\sigma^n$, we may assume that $\sigma(x) = x$.
 By Corollary~\ref{cor-1pt}, $p^{-1}(x) = \{v\}$ is a single $\kk$-point. Using Proposition~\ref{prop-RS}, let $N \geq n_0$ be such that $\Phi_N$ is a local isomorphism at $v$, and let $v_N: = \Phi_N(v)$.  Then $f_N^{-1}(x) = \{ v_N\}$.  
We will show that  $f_N$ is a local isomorphism at $v_N$:  that is, that $f_N^{-1}$ is defined at $x$.

Let $W_N$ be the (finite) indeterminacy locus of $f_N^{-1}$.  
Let $U_1 := Y'_N \smallsetminus f_N^{-1} (W_N \smallsetminus \{ x\})$.  Let $U_2
:= X \smallsetminus \{x\}$.  
Let $U_{12}:= U_1 \smallsetminus \{v_N\}$ and let $U_{21}:= X \smallsetminus W_N$.
 Let $\Xt$ be
the glueing of $U_1$ and $U_2$ along the isomorphism $f_N: U_{12}\to U_{21}$.  

For $j = 1, 2$ let $i_j:  U_j \to \Xt$ be the canonical map.  Since $i_1 = i_2
f_N$ as  a map from $U_{12}$ to $X$, we may glue $i_1$ and $i_2$ to obtain a
morphism $i:  \Xt \to X$ which is an isomorphism away from $i_1(v_N) \in \Xt$.

Let $U_3 := Y'_N \smallsetminus \{v_N\}$.  Then  $U_3 \cup U_1 = Y'_N$ and $U_3 \cap
U_1 = U_{12}$.  As 
\[ \xymatrix{ U_3 \ar[r]^{f_N} & U_2 \ar[r]^{i_2} & \Xt } \quad\quad \mbox{and} \quad\quad
\xymatrix{U_1 \ar[r]^{i_1} & \Xt}\]
agree on $U_{12}$, we may glue these maps to obtain a morphism $h:Y'_N \to \Xt$, with $ih = f_N$.  It follows from  Proposition~\ref{prop-II}(2) that $h|_{U_3}$ is scheme-theoretically surjective as a map from $U_3 \to U_2$, and so $h$ is scheme-theoretically surjective.

We next lift $\sigma$ to an automorphism $\wt{\sigma}$ of $\Xt$.
Let: 
\begin{multline*} U_1' = Y'_{N+1} \ssm f_{N+1}^{-1}( (W_{N+1} \cup \sigma^{-1} (W_N)) \ssm \{x\}), \\  
 U_2' = U_2 = X \ssm \{x\}, \\
  U_{12}' = Y'_{N+1} \ssm f_{N+1}^{-1}(  W_{N+1} \cup \sigma^{-1} (W_N)), \\ 
U_{21}' = X \ssm (  W_{N+1} \cup \sigma^{-1} (W_N)).
\end{multline*}
Just as above, we may glue $U_1'$ and $U_2'$ along the isomorphism $f_{N+1}: U_{12}' \to U_{21}'$.  
It is clear the resulting scheme is isomorphic to $\Xt$, and we obtain $i'_j:   U'_j \to \Xt$ for $j = 1,2$.

By construction, $\sigma$ induces an automorphism of $U_2'$ and $\psi(U_1') \subseteq U_1$.  It follows from the equation $\sigma f_{N+1} = f_N \psi_{N+1}$ that the diagram
\[ \xymatrix{
 U'_{12} \ar[r]^{\psi} \ar[d]_{f_{N+1}}^{\cong} & U_1 \ar[r]^{i_1} & \Xt \ar@{=}[d] \\
U'_{21} \ar[r]_{\sigma} & U'_2 \ar[r]_{i_2'} & \Xt}
\]
commutes, and so we may glue $i_1 \psi:  U'_1 \to \Xt$ and $i_2' \sigma:  U'_2 \to \Xt$ to obtain $\wt{\sigma}:  \Xt \to \Xt$. 

By Corollary~\ref{cor-1pt}, $\psi^{-1}$ is defined at $v_N$.  
Thus $(\wt{\sigma})^{-1}$ is defined at $v_N$.
As $\sigma$ induces an automorphism of $U'_2$, we obtain that $\wt{\sigma} \in \Aut(\Xt)$.
Clearly $h \psi_{N+1} = \wt{\sigma} h \phi_{N+1}$.

For $n \in \ZZ$, let $a_n = (\wt{\sigma})^{-n} h \Phi_N:  Y_\infty \to \Xt$.  We have $a_{n+m} \Psi^m = a_n$.
From the universal property of colimits and Lemma~\ref{lem:Hdefn}, there is an induced map  $a: \H \to \Xt$ so that
\[ \xymatrix{Y_{\infty} \ar[r]^{\pi} \ar[dr]_{a_0 = h \Phi_N } & \H \ar[d]^{a} \\
  & \Xt }
\]
commutes.

 By the universal property of $X$ there is thus a morphism $j: X \to \Xt$ so that 
\[ \xymatrix{
Y_{\infty}\ar[d]_p \ar[rrd]^(0.45){h \Phi_N} \ar[rr]^{\pi}&& \H \ar[d]^{a} \\
X \ar[rr]_j && \Xt}
\]
commutes.  
Clearly $ij = \id_X$, as $ij$ is the identity on closed points and $X$ is a variety.   As $h \Phi_N = jp$ is scheme-theoretically surjective, $j$ is also scheme-theoretically surjective.  By Lemma \ref{lem-surj-equal}(\ref{scholium}), $i$ and $j$ are isomorphisms and $X \cong \Xt$.  
In particular, both
$f_N^{-1} = (ih)^{-1}$ and $p^{-1} = (f_N \Phi_N)^{-1} $ are defined at $x$.  
\end{proof}

 To end the section, we  give an important ring-theoretic consequence of Theorem~\ref{thm-VIII}.

\begin{theorem}\label{thm-g}
Let $R$ be a connected graded noetherian algebra generated in degree 1 over an uncountable algebraically closed field $\kk$, and adopt Notation~\ref{not1}.   
Assume that  Hypotheses~\ref{hyp-FOO} hold.  
Let $K$ be the field of fractions of $X$.  Then there are an automorphism $\sigma$ of $K$ and a $\kk$-algebra homomorphism $g:  R \to K[t; \sigma]$ so that $\ker g$, $\ker \theta^e$, and $\ker \theta'$ are all equal.   
\end{theorem}
\begin{proof}
The automorphism $\sigma$ and the map $g$ are given by Proposition~\ref{prop-II}.

We first show that $\ker g$, $\ker \theta'$, and $\ker \theta^e$ are all equal in large degree.   
By Theorem~\ref{thm-universal1} we have $(\ker g)_{\gg 0} \supseteq (\ker \theta^e)_{\gg 0}$.  As $\ker \theta^e \supseteq \ker \theta'$, it suffices to prove that $\ker \theta'$ and $\ker g$ are equal in large degree.  
Let $\bbar{R}:= \theta'(R)$ and let $J := (\ker g + \ker \theta')/(\ker \theta') \subseteq \bbar{R}$.  
 For $n \geq n_0$, let $X_n$ be the strict
transform of $X$ in $Y'_n$.  Then 
\[ J_n = \{ h \in \bbar{R}_n \subseteq H^0(Y'_n, \sM'_n) \st h|_{X_n} \equiv 0\}.\]
Since $R$ is noetherian,  
for some $N \geq n_0$ there are $b_1, \ldots, b_k \in J_N$ that generate
$J_{\geq N} $ as both a right and a left ideal. 

Let $n \geq N$ and let $h\in J_n$, so $h|_{X_n}$ is identically 0.  Then there are
$a_1, \ldots, a_k, c_1, \ldots, c_k \in \bbar{R}_{n-N}$ so that
\[ h = \sum_i (\phi^{n-N})^*b_i (\psi^N)^* c_i = \sum_i (\phi^N)^* a_i (\psi^{n-N})^* b_i.\]
Now, if $y \in Y'_n$ with $ f_n(y)\not\in W_N$, then $f_N$ is a local isomorphism at $\phi^{n-N}(y)$.  Thus each $b_i$ vanishes in a  neighborhood of $\phi^{n-N}(y)$, and so $(\phi^{n-N})^* b_i$ vanishes identically along $(\phi^{n-N})^{-1} \phi^{n-N}(y) = f_n^{-1} f_n(y)$.  That is, from the first equation we conclude that $h$ vanishes at all points in $Y'_n \ssm f_n^{-1}(W_N)$.  
  Likewise, it follows from the second equation that $h$ vanishes away from 
$f_n^{-1}(\sigma^{-(n-N)}(W_N))$.    Since by Theorem~\ref{thm-VIII} all points in $W_N$ have infinite $\sigma$-orbits, for $n \gg 0$ the
set $W_N \cap \sigma^{-(n-N)}(W_N) = \emptyset$, and all $h \in J_n$ must vanish identically on $Y'_n$:  that is, $J_n = 0$ for $n \gg 0$. 

We now claim that $\ker g = \ker \theta'$.  
 Let $A$ be a commutative $\kk$-algebra, and let $M = R_A/I$ be an $R_A$-point module.  By Corollary~\ref{cor-ann}, $I \supseteq \ker \theta'$, and by the above $\ker \theta'$ and $\ker g$ are equal in large degree. Thus  there is some $N \in \NN$ so that $I \supseteq (\ker g)_{\geq N}$.  Thus $(I + \ker g)/I$ is a torsion submodule  of $ M_R$.   But we have:
\begin{sublemma}\label{sublem-tf}
 Let $R$ be a noetherian $\kk$-algebra generated in degree 1, let $A$ be a commutative $\kk$-algebra, and let $M$ be an $R_A$-point module.  Then $M_R$ is torsion-free.
\end{sublemma}
\begin{proof}
Let $m \in M_i$, and suppose that $m R_k = 0$.  By shifting, we may suppose that $i=0$, and as $M$ is an $R_A$-point module we may identify $M_0$ with $A$ and let $m \in A$.  By abuse of notation, we let $1 \in M_0$ be the generator of $M$.  Then we have $0 = m R_k A= 1\cdot m R_k A= 1 \cdot R_k A m = M_k m$.  As $M_k \cong A$, we have $m =0$.
\end{proof}

Returning to the proof of Theorem~\ref{thm-g}, the torsion submodule $(I+\ker g)/I$ of $M_R$ is 0, so $\ker g \subseteq I$.

As $I = \Ann_{R_A}(M_0)$ and $M$ was arbitrary, we have  
\[ \ker g \subseteq  \bigcap \{ \Ann_R(M_0) \st M \mbox{ is an $R_A$-point module for some commutative $\kk$-algebra $A$ } \}.\]
This  is $\ker \theta'$ by Corollary~\ref{cor-ann}.  Thus we have
$\ker \theta'  \subseteq \ker g \subseteq \ker \theta'$ and the two are equal.

We thus have $\ker g = \ker \theta' \subseteq \ker \theta^e$, and the three are equal in large degree.  But $g(R) \subseteq K[t; \sigma]$, and this is a domain; hence the torsion ideal $(\ker \theta^e)/(\ker g)$ of $g(R)$ must be 0, and $\ker \theta^e = \ker g = \ker \theta'$.
\end{proof}

\begin{corollary}\label{cor:birational factor}
 Assume  Hypotheses~\ref{hyp-FOO} hold, and let $g:  R \to K[t; \sigma]$ be the algebra homomorphism constructed in Proposition~\ref{prop-II}.
Then $g$ is universal for maps from $R$ to birationally commutative algebras.  
In particular, the canonical birationally commutative factor of $R$ is birationally commutative.
\end{corollary}
\begin{proof}
 Since $\ker g = \ker \theta^e$, this is just the universal property of $\theta^e$ from Theorem~\ref{thm-universal1}.
\end{proof}

\section{Defining data for the canonical birationally commutative factor}\label{DATA}
We now want to understand the canonical birationally commutative factor $g(R)$ of $R$.  We will see in the next section that $g(R)$ is equal in large degree to a \nba\ $R(X, \sL, \sigma, P)$.
We have already constructed $\sL$ and $\sigma$, and in this section we show that $\sL$ is invertible and $\sigma$-ample.  We construct the 0-dimensional subscheme $P$ that is the missing piece of data for the na\"ive blowup, and show that $g(R) \subseteq R(X, \sL, \sigma, P)$.  We show also that  the points in $P$ have dense orbits.  (In the final section, we prove that they have critically dense orbits, and that the indeterminacy locus of $p^{-1}$ is supported on these orbits.)
Throughout, we assume that the hypotheses of Theorem~\ref{thm-main} hold.  

We begin by showing that $\sL$ is invertible.
We need a further piece of notation.

\begin{notation}\label{not-ZW}
Let $R$ be a connected graded algebra generated in degree 1, and assume that the hypotheses of Theorem~\ref{thm-main} hold; further adopt Notation~\ref{not3}.  
For $n \geq n_0$ let $X_n \subseteq Y_n$ be the strict transform of $X$, and let $X_\infty := \varprojlim X_n \subseteq Y_\infty$.  Let $W'_n \subseteq W_n $ be the locus in $X$ where $(f_n|_{X_n})^{-1}:  X \dra X_n$ is undefined, and let $\mb{W} := \bigcup_{n \geq n_0} W'_n$.  
Note that $\mb{W}$ is the locus where $(p|_{X_\infty})^{-1}:  X \dra X_{\infty}$ is undefined.
\end{notation}

It is immediate that $\psi_n(X_n) = X_{n-1}$ and $\Psi(X_\infty) = X_\infty$.  Thus $X_\infty \subseteq Y_\infty^e$.

\begin{lemma}\label{lem:surface implies points in P are smooth}
 If $X$ is a surface, then $X$ is nonsingular at all points in $\mb{W}$.  
\end{lemma}
\begin{proof}
 Let $z \in X \ssm \mb W$.  For any $n \geq n_0$, consider the commutative diagram
\[ \xymatrix{ \Spec \sO_{X, z} \times_X X_{n+1} \ar[r]^{\psi} \ar[d]_{\cong} & \Spec \sO_{X, \sigma(z)} \times_X X_n \ar[d]^{f_n} \\
\Spec \sO_{X, z} \ar[r]_{\sigma}^{\cong} &  \Spec \sO_{X, \sigma(z)}. }
\]
The top map is  (scheme-theoretically) surjective.
By Lemma \ref{lem-surj-equal}(\ref{scholium}), the top and right-hand maps are isomorphisms, and   $(f_n|_{X_n})^{-1}$ is defined at $\sigma(z)$.
Since $n$ was arbitrary,  $\sigma(z) \not\in \mb W$.  It follows that   $\sigma^{-1}(\mb W) \subseteq \mb W$.

Let $X^{\rm sing}$ be the singular locus of $X$.  Write $X^{\rm sing} = X^{(1)} \cup X^{(2)}$, where $X^{(1)}$ is a (possibly reducible or empty) $\sigma$-invariant curve and $X^{(2)}$ is 0-dimensional and $\sigma$-invariant.  
Let $w \in \mb{W} \cap X^{\sing}$.  By Theorem~\ref{thm-VIII} the $\sigma$-orbit of  $w$ is infinite, so  $w \in X^{\rm sing} \ssm X^{(2)} \subseteq  X^{(1)}$.

As $X^{(1)}$ is $\sigma$-invariant, 
\[ \{ \sigma^{-k}(w) \st k \geq 0\} \subseteq X^{(1)} \cap \mb{W},\]
which contradicts Corollary~\ref{cor-curves}.  Thus no such $w$ can exist.
\end{proof}

\begin{corollary}\label{cor:L is invertible}
 Adopt Notation~\ref{not3}, and assume that the hypotheses of Theorem~\ref{thm-main} hold.
Then $\sL$ is invertible, and $g(R) \subseteq B(X, \sL, \sigma)$.
\end{corollary}
\begin{proof}
 By Proposition~\ref{prop-N-L}(5) it suffices to prove that $\sN_n = (\sR'_n)^{\vee\vee}$ is invertible for $n\gg 0$.  Adopt Notation~\ref{not-ZW}.  
Let $n \geq n_0$, and let $\sR''_n$ be the image of the natural map $\sR'_n \to \sN_n$. 
Then $\sR''_n$ is the $X$-torsion-free part of $\sR'_n$. 
Let $x \in X$. 
If $x \not \in W'_n$, then $X_n$ is locally isomorphic to $X$ at $x$, and so $(\sR''_n)_x = (\sN_n)_x$ is invertible.

Now suppose that $x \in W'_n$.  
If $X$ is locally factorial at points of indeterminacy of $p^{-1}$, then $\sN_n = (\sR'_n)^{\vee\vee}$ is invertible by definition. 
If $X$ is a surface then 
 $X$ is nonsingular at $x$ by Lemma~\ref{lem:surface implies points in P are smooth}, and so locally factorial at $x$.  
Again $(\sN_n)_x$ is invertible.

That $g(R) \subseteq B(X, \sL, \sigma)$ follows by construction.
\end{proof}

Now that we know $\sL$ is invertible, we can show that  $\sL$ is appropriately positive:  that is, $\sL$ is  $\sigma$-ample.  
The condition that $\sL$ is $\sigma$-ample is technical but important:  it means that the twisted tensor powers $\sL_n$ have the same good cohomological properties as the powers of an ample invertible sheaf.  More formally, 
we say that 
a sequence $\{\sS_n\}$ of bimodules on a scheme $X$ is {\em left }
(respectively, {\em right}) {\em ample} if for every $j \geq 1$ and every
coherent sheaf $\sM$ on $X$ we have
$H^j(X, \sS_n \otimes \sM)=0$ (respectively, $H^j(X, \sM \otimes \sS_n) =0$) for
$n
\gg 0$.   An invertible sheaf $\sL$ is {\em $\sigma$-ample} if $\{(\sL_n)_{\sigma^n}\}$ is left (equivalently,  by \cite[Theorem~1.2]{Keeler2000}, right) ample.  

Let $K$ be the function field of $X$, and let $\sigma$ denote also the element of $\Aut_\kk(K)$ that acts via pullback by $\sigma \in \Aut_\kk(X)$. Recall that a
graded subalgebra $S$ of $K[t; \sigma]$ is {\em big} if there is some $n
\geq 1$ and some $u \in S_n$ so that $K$ is generated by $S_n u^{-1}$ (as a field); morally, we want $K[t, t^{-1};\sigma]$ to be the graded  quotient ring of $S$.

\begin{theorem}\label{thm-ample}
Assume that the hypotheses of Theorem~\ref{thm-main} hold, and further adopt Notation~\ref{not-ZW}.  Then $\sL$ is $\sigma$-ample.
\end{theorem}

\begin{proof}
Fix $n\geq n_0$.    
Let $f:= f_n |_{X_n}$, so $f$ is birational, and let $\sM := \sM_n |_{X_n}$.  
There are natural maps
\[\xymatrix{
H^0(X_n, \sM)\otimes \sO_{X_n} \ar[r] & f^* f_* \sM \ar[r] & \sM.}
\]
Since $\sM$ is globally generated, the composition is surjective, and thus  $f^* f_* \sM \to \sM$ is surjective.  It is an isomorphism at the generic point of $X_n$, and so $\sM$ is the torsion-free part of $f^* f_* \sM$.

There is a morphism of sheaves $\sR'_n = (f_n)_* \sM'_n \to f_* \sM$ whose kernel and cokernel are supported on $W_n$.  
We have  $ (f_* \sM)^{\vee\vee}\cong (\sR'_n)^{\vee\vee} \cong \sL_n $.  Pulling back the canonical map $f_* \sM \to (f_* \sM)^{\vee\vee}$, we obtain a (nonzero) map $f^* f_* \sM \to f^* \sL_n$.  Since $f^* \sL_n$ is invertible and thus torsion-free, by the previous paragraph we obtain an induced injective map $\alpha:  \sM \to f^* \sL_n$ so that
\[ \xymatrix{
f^* f_* \sM \ar[r] \ar[rd] & \sM \ar[d]^{\alpha} \\
& f^* \sL_n  }
\]
commutes.   Let $\sT := f^* \sL_n \otimes \sM^{-1}$.  Then $\sT$ is effective and as a divisor is supported on an $f$-exceptional divisor, say $T$.

We claim next that $\sL_n$ is ample.  We will use terminology and results from \cite{Laz} on big and nef (Cartier) divisors, and  from \cite{Fulton1998} on intersection theory; in particular, if $V$ is a $k$-dimensional subvariety of $X$, let $[V]$ denote the associated $k$-cycle in $A_k(X)$, the group of $k$-cycles modulo rational equivalence.  Let $f_*:  A_k(X_n) \to A_k(X)$ denote  pushforward on cycle classes as in \cite{Fulton1998}.  We  identify  $A_0(X)$ and $A_0(X_n)$ with $\ZZ$; thus $f_*:  A_0(X_n) \to A_0(X)$ is the identity on $\ZZ$.  

Let $V \subseteq X$ be a subvariety of dimension $d >0$.  Let $\Vt$ be its strict transform in $Y_n$.  
As $T$ is $f$-exceptional, $\Vt \not\subseteq T$.  Thus $\sT|_{\Vt}$ is still effective, and  as $\sM$ is ample, $(f^* \sL_n)|_{\Vt} \cong \sM|_{\Vt} \otimes \sT |_{\Vt}$ is big by \cite[Corollary~2.2.7]{Laz}.  (We remark that, although \cite{Laz} is written with 
characteristic zero hypotheses throughout, the proof of Corollary~2.2.7 there does not use any hypothesis on the characteristic and the result holds in general.)

We claim that $f^* \sL_n$ is nef on $X_n$, and therefore its restriction is nef on $\Vt$.  To see this, let $C$ be an irreducible curve on $X_n$.  We must show that $c_1(f^* \sL_n). [C] \geq 0$.  
If $f$ contracts $C$ to a point, then $c_1(f^* \sL_n) . [C] = 0$.  Otherwise, $C$ is the strict transform of $f( C)$, and by the previous paragraph $f^* \sL_n |_C$ is big.  It thus has positive degree, and so $ c_1(f^* \sL_n) .[C] >0$.  

Thus $(f^* \sL_n)|_{\Vt}$ is big and nef.  By
\cite[Theorem~2.2.16]{Laz} (which works in arbitrary characteristic---cf. also \cite[Definition-Lemma~0.0]{Keel1999}),
$\int_{\Vt} (c_1 (f^* \sL_n))^d > 0$.   By repeated applications of the projection formula \cite[Proposition~2.5(c)]{Fulton1998}, we obtain
\[
\int_{V} (c_1 ( \sL_n))^d = (c_1 ( \sL_n))^d . f_*[ \Vt] = f_* \bigl( (c_1(f^* \sL_n))^d .[ \Vt]\bigr) = \int_{\Vt} (c_1 (f^* \sL_n))^d>0.
\]

In summary, we have established that $\int_{V} (c_1 ( \sL_n))^{\dim V} > 0 $ for any irreducible subvariety $V$ of $X$.  The Nakai criterion \cite[Theorem~III.1]{Kleiman1966} thus implies that $\sL_n$ is ample.  

It remains to conclude that $\sL$ is $\sigma$-ample.  
We   write $g(R) \subseteq B(X, \sL, \sigma) \subseteq K[t; \sigma]$, and may assume without loss of generality that $t \in g(R_1)$.   By Theorem~\ref{thm-g}, we may identify   $g(R_n)$ with $ \theta^e(R_n)$.  This set of sections embeds $X_n$ in some projective space, and as $X_n$ is birational to $X$ the vector space $g(R_n) t^{-n}$ generates $K$.  Thus  $g(R)$ is a big subalgebra of $K[t; \sigma]$.  As $g(R)$ is noetherian, by \cite[Theorem~0.1]{SteZh}, $g(R)$ has subexponential growth.  

The automorphism $\sigma$ acts naturally on the vector space $\Num(X)$ of divisors modulo numerical equivalence.   By \cite[Proposition~3.5]{RZ2008}, this action must be {\em quasi-unipotent}:  that is, all eigenvalues have modulus 1.  By \cite[Theorem~1.3]{Keeler2000}, $\sL$ is $\sigma$-ample.
\end{proof}

We next prove a lemma on the preimages of points in $\mb{W}$.

\begin{lemma}\label{lem-istaristar}
 Adopt Notation~\ref{not-ZW} and assume that the hypotheses of Theorem~\ref{thm-main} hold.
Let $n \geq n_0$ and let $f := f_n |_{X_n}: X_n \to X$.
Then $f_* \sO_{X_n} = X$.  
Further, if $w \in W'_n$ then $f^{-1}(w)$ is positive-dimensional.
\end{lemma}
\begin{proof}
 Let $X'_n := \sS\!\it{pec}_X f_* \sO_{X_n}$ and consider the Stein factorization
\[ \xymatrix{
 X_n \ar[r]^{a} \ar[rd]_{f} & X'_n  \ar[d]^{b} \\
& X}
\]
of $f$.
If $X$ is a surface, by Lemma~\ref{lem:surface implies points in P are smooth} it is normal at all points in $W'_n$.  This also follows if $X$ is locally factorial at points in $W_n$.  In either case  the finite birational map $b$ is  a local isomorphism above $W'_n$, and $f^{-1}(w)$ must be positive-dimensional.

Above points not in $W'_n$, both $a$ and $b$ are local isomorphisms by definition.  Thus $b$ is a local isomorphism at all points and so an isomorphism.
\end{proof}

We note that this result may be proved without the local factoriality assumption on $X$, merely using Theorem~\ref{thm-VIII}.
We do not include this proof here.

We now define the final piece of data we need to define the \nba\ $R(X, \sL, \sigma, P)$. 
This will allow us to show that $\mb W$ is supported on finitely many dense $\sigma$-orbits.

\begin{notation}\label{not-final}
 Assume that the hypotheses of Theorem~\ref{thm-main} hold, and adopt Notation~\ref{not-ZW}.
For $n \in \NN$, let $\sR_n \subseteq \sL_n$ be the subsheaf generated by the sections in $g(R_n)$.
Let $P$ be the base locus of $g(R_1)$ --- that is, define $P$ by $\sI_P = \sR_1 \sL^{-1}$.
\end{notation}

\begin{lemma}\label{lem-locus}
 Assume that the hypotheses of Theorem~\ref{thm-main} hold, and adopt Notation~\ref{not-ZW}.
\begin{enumerate}
 \item We have $g(R) \subseteq R(X, \sL, \sigma, P)$.
\item 
Given $n \in \NN$, define
\[ P_n := \Supp( P \cup \sigma^{-1}(P) \cup \cdots \cup \sigma^{-(n-1)}(P)).\]
Then  $P_n = W'_n$  for $n \geq n_0$, and set-theoretically we have $\mb W = \bigcup_{k \geq 0} \sigma^{-k}(P)$.  
\end{enumerate}
\end{lemma}
\begin{proof}
$(1)$. Since $\sR_n$ is generated by $g(R_n) = g(R_1)^n$, we have  $\sR_n = \sR_1 \sR_1^{\sigma} \cdots \sR_1^{\sigma^{n-1}}$.  $(1)$ follows immediately.

$(2)$.  We see that $\sR_n$ is invertible at $x$ if and only if each of the $\sR_1^{\sigma^i}$ are invertible at $x$, for  $0 \leq i \leq n-1$:  that is, $P_n$ is exactly the locus where $\sR_n$ is not invertible.  

Let $f:= f_n|_{X_n}$.  
Let $w \in W'_n$ and let $T := f^{-1}(w) \subseteq X_n$. 
By Lemma~\ref{lem-istaristar}, $\dim T \geq 1$.
Let $r \in R_n \ssm \ker \theta'$.  Then $\theta'(r)$ is a nonzero element of $H^0(Y'_n, \sM'_n)$.  
Because  $\dim T \geq 1$ and $\sM'_n$ is very ample, $\theta'(r)$ must vanish at some point of $T$.
Thus $g(r)$, regarded as a section of the invertible sheaf $\sL_n$, must vanish at $w$ and is contained in 
$H^0(X, \sI_w \sL_n)$.
We see that $\sR_n \subseteq \sI_w \sL_n$, and $w \in P_n$.

Now suppose that $x\not\in W'_n$.  
Let $\sR''_n$ be the image of the natural map $\sR'_n \to \sL_n$.
Since  $f$ is a local isomorphism above $x$, therefore $(\sR''_n)_x = (\sL_n)_x$ is invertible.

The sections in (the image of) $R_n$  generate $\sM_n |_{X_n}$ at all points of $X_n$.  
Thus their images under $g$ must locally generate $\sR''_n$ at $x$, and so
\[ (\sR_n)_x = (\sR''_n)_x = (\sL_n)_x,\]
and $x \not \in P_n$.

Thus $P_n = W'_n$.  
It follows that $\mb W = \bigcup_{k \geq 0} \sigma^{-k}(P)$.
\end{proof}

Our next goal is to show that all points of $P$ have dense orbits.  
Before proving this result, we need a lemma. 

\begin{lemma}\label{lem-ample}
Let $X$ be a projective scheme of dimension $\geq 2$, let $\sigma \in \Aut_\kk X$,
and let $\sL$ be a $\sigma$-ample invertible sheaf on $X$.  Let $\sI$ be an
ideal sheaf on $X$, and let $P$ be the subscheme defined by $\sI$.  For $n \geq
0$, let  $\sS_n :=  \sI \sI^{\sigma} \cdots
\sI^{\sigma^{n-1}}  \sL_n$.

Suppose that $(1)$ $\dim P = 0$ and $(2)$ for every $x \in \Supp P$ and every irreducible 
component $Y$ of $X$ so that $\{\sigma^n( x) \}$ meets $Y$, the set $\{\sigma^n (x)\}
\cap
Y$ is Zariski-dense in $Y$.    Then $\{(\sS_n)_{\sigma^n}\}$ is a left and right ample
sequence on $X$.
\end{lemma}
\begin{proof}
By symmetry, it
suffices to prove that the sequence is right ample.  
By replacing $\sigma$ by a power, we may  suppose that each irreducible component
of $X$ is $\sigma$-invariant.

Let $X^{(i)}$ denote the $i$th infinitesimal neighborhood of $X_{\rm red}\subseteq X$; then $X^{(N)} = X$ for 
some $N$.  We first reduce to showing that $\{\sS_n|_{X_{\rm red}}\}$ is an ample sequence.  
 Let $\sM$ be a coherent sheaf supported on $X^{(m)}$.  Then there is an exact
sequence
\[ 0 \to \sM' \to \sM \to \sM'' \to 0,\]
where $\sM''$ is supported on $X^{\red}$ and $\sM'$ is supported on $X^{(m-1)}$.
Fix $j \geq 1$.  For any $n$, there is an exact sequence
\[ 0 \to \sK \to \sM' \otimes \sS_n \to \sM \otimes \sS_n \to \sM'' \otimes \sS_n \to
0,\]
where $\sK$ is supported on a set of dimension 0.  Assume, by way of induction hypothesis, that $H^j(X,
\sM'\otimes \sS_n)$ and $H^j(X, \sM'' \otimes \sS_n)$ vanish for $j \geq 1$ and $n \gg 0$. 
Then $H^j(X, \sM \otimes \sS_n) =0$ for $j \geq 1$ and $n \gg 0$.  Taking $m=N$, we conclude that if
$\{\sS_n|_{X_{\rm red}}\}$ is an ample sequence then so is $\{\sS_n\}$.  

Now suppose that $X = X_{\rm red}$ is  reduced 
and let $\{X_c\}_{c=1, \dots, \ell}$ be a list of the irreducible components of $X$.  Suppose that $\sM$ is a coherent sheaf supported on $X_1\cup\dots\cup X_j$ (for some $1\leq j\leq \ell$).  We have an exact sequence 
\begin{displaymath}
0\rightarrow \sM'\rightarrow \sM\rightarrow \sM|_{X_j}\rightarrow 0
\end{displaymath}
with $\sM'$ supported on $X_1\cup\dots\cup X_{j-1}$.  As all points of $P$ have dense orbits in their components, no translate of any point in $P$ can lie
in the intersection of two components.  Thus the sequence 
\begin{displaymath}
0\rightarrow \sM'\otimes \sS_n\rightarrow \sM\otimes \sS_n\rightarrow \sM|_{X_j}\otimes \sS_n\rightarrow 0,
\end{displaymath}
obtained by tensoring with $\sS_n$ remains exact.  An induction thus reduces us to showing that $\{ \sS_n|_{X_j} \}$ is a right ample sequence for each $j$:  that is, it suffices to prove the result in the case that  $X$ is both irreducible and reduced.  But this  is precisely
\cite[Theorem~3.1(1)]{RS-0}.  
\end{proof}

\begin{proposition}\label{prop-Zdense}
Assume that the hypotheses of Theorem~\ref{thm-main} hold, and further adopt Notation~\ref{not-final}.  Then any $w \in P$ has a dense orbit in $X$. 
\end{proposition}
\begin{proof}
Suppose that for some $w \in P$, the orbit closure  $\Gamma := \bbar{\{\sigma^n(w)\}} \neq X$.  Now
$\Gamma$ is $\sigma$-invariant, and by passing to a  Veronese and replacing $\sigma$ by a power, we may
assume that $\Gamma$ is irreducible.  We may also assume that $\Gamma$ is
minimal:  i.e. any $y \in P \cap \Gamma$ has a dense orbit in $\Gamma$.  By
Corollary~\ref{cor-curves},  $\dim \Gamma \geq 2$.  

 Since the cokernel of the maps $\sR_n \to \sL_n$ is 0-dimensional, and all points in $P$ and $W'_n$ have infinite $\sigma$-orbits, it is an easy exercise that there is some $k \in \NN$ so that 
\beq\label{supset} (\sR_m)_x \supseteq (\sI_{\Gamma}^{(k)} \sL_m)_x\eeq
for all $x \in \Gamma$ and for all $m \geq n_0 \in \NN$.  (Here $\sI_{\Gamma}^{(k)}$ is the $k$'th symbolic power of $\sI_{\Gamma}$; we refer to the subscheme of $X$ it defines as $k\Gamma$.)  By increasing $k$, we may assume \eqref{supset} holds for   all $m \in \NN$.  

Let 
 $ \Gamma' := (k+1)\Gamma$ be the subscheme of $X$ defined by $\sI_{\Gamma}^{(k+1)}$.    Let $\sF := \sL|_{\Gamma'}$. 
Now, $\Gamma'$ is $\sigma$-invariant; we abuse notation and let $\sigma$ also denote $\sigma|_{\Gamma'}$. 
By  Theorem~\ref{thm-ample}, $\sF$ is $\sigma$-ample on $\Gamma'$.

For each $n$, let $\sS_n \subseteq \sF_n$ be the image of $\sR_n$ in $\sF_n$.  Since $\sF_n$ is invertible, we may define $\sJ_n := \sS_n (\sF_n)^{-1} \subseteq \sO_{\Gamma'}$.  We have $\sJ_n = \sJ_1 \sJ_1^{\sigma} \cdots \sJ_1^{\sigma^{n-1}}$.  

Let $S:= \bigoplus H^0( \Gamma', \sS_n)$.  
Since  the points in the subscheme defined by $\sJ_n$ have dense orbits
in $\Gamma$, by Lemma~\ref{lem-ample} the sequence of bimodules
$ \{(\sS_n)_{\sigma^n}\}$
is left and right ample on $\Gamma'$.  By \cite[Lemma~7.4]{S-surfclass}, $S$ is a
finitely generated left and right module over $R$.  Thus $S$ is left and right
noetherian.   

Let $\sH$ be the ideal sheaf on $\Gamma'$ defining $k \Gamma \subseteq \Gamma'$.  We have $\sJ_n \supseteq \sH$ for all $n \in \NN$, by \eqref{supset}. Let 
\[ H:= \bigoplus_{n \geq 0} H^0( \Gamma', \sH \sF_n  ).\]
Since $\sH \sF_n \subseteq \sS_n$ for all $n$ and 
$k \Gamma$ is
$\sigma$-invariant, this is a two-sided ideal of $S$.  

By Lemma~\ref{lem-locus}, $\sJ_n \subseteq \sI_w$ for all $n \in \NN$.
Let $m \geq 1 \in \NN$.  
By Nakayama's lemma,  
 $(\sH \cap  \sJ_n) \sJ_m^{\sigma^n} = \sH \sJ_m^{\sigma^n} \subsetneqq \sH = \sH \cap \sJ_{n+m}$ for any $n \geq 1$.  
 Since $\sF$ is $\sigma$-ample, for $n \gg 0$ the sheaf
$\sH \sF_{n+m}$ is globally generated.  Thus for $n \gg 0$, we have
$ H_n  S_m \subseteq H^0(\Gamma', (\sH \sF_{n} ) \sS_m^{\sigma^n}) \subsetneqq H_{n+m}$.
This shows $H$ is not finitely generated as a right ideal, giving a
contradiction. 

Thus the orbit of  $w$ is dense  in $X$.
\end{proof}

\section{The main theorem}\label{PROOF}

In this section, we prove Theorem~\ref{thm-main}; in fact, we prove a stronger structure theorem for $R$ and for the point space $Y_{\infty}$.

\begin{theorem}\label{thm-final}
Let $R$ be a  noetherian connected graded $\kk$-algebra generated in degree 1, where $\kk$ is an algebraically closed uncountable field.  Let $Y_{\infty}$
be the point space of $R$, and let $F$, $\H$ be as in  Notation~\ref{not1}.  
Suppose:
\begin{itemize}\item[(i)] there is a commutative
diagram
\[ \xymatrix@C=5pt{
& F \cong  Y_{\infty} \ar[ld]_{\pi} \ar[rd]^{p} \\
\H \ar[rr] && X }\]
 where $X$ is a  projective scheme that corepresents $\H$ through the morphism $H \longrightarrow X$. 
\end{itemize}

\noindent  Suppose further that:
\begin{itemize}
	\item[(ii)] $X$ is a variety of dimension $\geq 2$ that is either a surface or locally factorial at all indeterminacy points of $p^{-1}$;
	\item[(iii)] the map $\H \to X$ is bijective on $\kk$-points;
	\item[(iv)] the indeterminacy locus of $p^{-1}$ consists  (set-theoretically) of
	countably many points;
\end{itemize} 
Then:
	\begin{itemize}
\item[(a)] $Y_{\infty}$ is a noetherian fpqc-algebraic space;
\item[(b)] Let $\Omega$ be the indeterminacy locus of $p^{-1}$.  Then  there are an automorphism $\sigma$ of $X$ and  a finite subset $P'$  of $X$, all of whose points have critically dense orbits, so that $\Omega$ is supported on  half-orbits of points in $P'$.  In particular, $\Omega$ is critically dense in $X$.
\item[(c)]  there are a $\sigma$-ample invertible sheaf $\sL$ on $X$ and a closed subscheme $P$ of $X$ with $\Supp P=P'$ so that there is a homomorphism, surjective in large degree,
\[ g:  R \to R(X, \sL, \sigma, P).\]

\noindent Furthermore:

\item[(d)]  any graded homomorphism from $R$ to a birationally commutative algebra factors through $g$ in large degree.  Thus, up to finite dimension $R(X, \sL, \sigma, P)$ is the canonical birationally commutative factor of $R$. 
\item[(e)] Define $\theta', \theta^e$ as in Section~\ref{PPR}.  Then $ \ker g = \ker \theta' = \ker \theta^e$.
\end{itemize}
\end{theorem}
\begin{proof}
(e).  
Let $K := \kk(X)$.  Let $\sigma \in \Aut_\kk(X) \subseteq \Aut_{\kk}(K)$ and $g:  R \to K[t; \sigma]$ be given by  Theorem~\ref{thm-g}.  We write $S:= g(R)$.
By Theorem~\ref{thm-g}, we have $\ker g = \ker \theta' = \ker \theta^e$.

(d) is   the universal property of $\theta^e$ in Theorem~\ref{thm-universal1}.

(c)
Let the rank 1 reflexive sheaf $\sL$ on $X$ be as in Notation~\ref{not3}.  By Corollary~\ref{cor:L is invertible}, $\sL$ is  invertible  and $S \subseteq B(X, \sL, \sigma)$.
By Theorem~\ref{thm-ample}, $\sL$ is $\sigma$-ample.

For $n \geq 0$ define $\sR_n \subseteq \sL_n$ to be the subsheaf generated by the sections in $g(R_n)$.  Let $P$ be the subscheme of $X$ defined by $\sR_1 \sL^{-1}$.    
Let $P' := \Supp P$.
 Since $g(R)$ is generated in degree 1, we have $\sR_n = \sR_1 \sR_1^{\sigma} \cdots \sR_1^{\sigma^n}$, and $g(R) \subseteq  R(X, \sL, \sigma, P) \subseteq  B(X, \sL, \sigma)$.    

Adopt Notation~\ref{not-ZW}.  That is, for  $n \geq n_0$, let $X_n \subseteq Y'_n$ be the strict transform of $X$.  Recall that  $W'_n $ is the indeterminacy locus of the rational map $(f_n |_{X_n})^{-1}:  X \dra X_n$.   By Lemma~\ref{lem-locus}, we have  $W'_n = P' \cup \cdots \cup \sigma^{-(n-1)}(P')$ for $n \geq n_0$.  
By Proposition~\ref{prop-Zdense}, all points in $P'$ have dense orbits.

The sections in $g(R_n) = \theta'(R_n)$ define an immersion of $X \ssm W'_n \cong f_n^{-1}(X \ssm W'_n ) \cap X_n \hra \PP^d$ for appropriate $d$.  
We apply the following result of Rogalski and Stafford:

\begin{theorem}\label{thm-surjcrit}
{\rm (\cite[Theorem~9.2]{RS-0})}
Let $X$ be a  projective variety, let $\sigma \in \Aut(X)$, and let
$\sL$ be a $\sigma$-ample invertible sheaf on $X$.  Let $\sI = \sI_P$ be an
ideal sheaf on $X$ so that $P$ is 0-dimensional and so that all points in $P$ have dense orbits.

Let $S' \subseteq R(X, \sL,
\sigma, P)$ be a noetherian subalgebra so that
\begin{enumerate}
\item for $n \gg 0 $ the sections in $S'_n $ generate $\sI_P(\sI_P)^\sigma \dots (\sI_P)^{\sigma^{n-1}} \sL_n$; and
\item for $n \gg 0$ the sections in $S'_n $ restrict to give an immersion 
\[ X \ssm (P \cup \sigma^{-1}(P) \cup \cdots \cup \sigma^{-(n-1)}(P)) \hra \PP^N.\]
\end{enumerate}
Then $S' = R(X, \sL, \sigma, P)$ up to finite dimension, and all points in $P$
have critically dense orbits.   
\end{theorem}

We obtain immediately that $S= g(R) $ and $R(X, \sL, \sigma, P)$  are equal in large degree, and that all points in $P'$ have critically dense orbits.

(a).   By part (e), $\ker g$ annihilates any $R$-point module. Thus $Y_\infty$ is also the point space for $S=g(R)$, and is a noetherian fpqc-algebraic space by Theorem~\ref{thm:NS1}.

(b)    Let $q: Y_\infty \to X$ be the map constructed in Proposition~\ref{prop:existence of q}.  
By that result, $q$ factors through $\H$; as $\H \to X$ corepresents $\H$, thus $q$  factors through $p$, and there is a morphism $\tau: X \to X$ so that $\tau p = q$.  
By Proposition~\ref{prop:existence of q}(iii), $q^{-1}$ is defined at the generic point of $X$.
As $p^{-1}$ is defined at the generic point of $X$ by assumption,  $\tau$ is birational and thus:
\begin{claim}
$\tau$ is an automorphism of $X$.
\end{claim}
\begin{proof}[Proof of Claim.]
As in the proof of Lemma~2.3(1) of \cite{Fujimoto2002}, $\tau$ is finite, as well as birational. (We note  the cited result applies to  projective nonsingular varieties defined over $\CC$; but the latter two hypotheses are not used in the proof we cite.)
By finiteness of the integral closure,  $(\tau^{n})_* \sO_X = (\tau^{n+1})_* \sO_X$ for some $n\in \NN$.  
As $\tau^n$ is finite and so affine, we have $\sO_X = \tau_* \sO_X$, and so $\tau$ is an isomorphism.
\end{proof}
Returning to the proof of the theorem, 
let $\Omega$ be the indeterminacy locus of $p^{-1}$; then $\tau( \Omega)$ is the indeterminacy locus of $q^{-1}$.   By Proposition~\ref{prop:existence of q}(iii), we have
\[ \tau (\Omega) = \bigcup \{ \sigma^{-n}( P') \st n \geq 0\}.\]

Consider the maps
\[ \xymatrix@R=5pt{
 X \ar[rr]^{\tau} \ar[dd]_{\sigma} && X \ar[dd]^{\sigma} \\
& Y_\infty \ar[lu]^{p} \ar[ru]_{q} \ar[dd]_(.3){\Psi} \\
X \ar'[r] [rr]^{\tau} && X \\
& Y_{\infty}. \ar[lu]^{p} \ar[ru]_{q} }
\]
The top and  bottom faces of the diagram commute by definition of $\tau$, and the  left  face  commutes by the construction of $\sigma$ in Proposition~\ref{prop-II}.    
The right face commutes by Proposition~\ref{prop:existence of q}(\dag).
Thus 
\[ \sigma \tau p =  \sigma q = q \Psi = \tau p \Psi = \tau \sigma p.\]
By Lemma~\ref{lem-surj-equal} we have
\beq \label{stcommute} \sigma \tau = \tau  \sigma. \eeq

Note that $p^{-1} $ is not defined at any point in $\bigcup_n W'_n = \bigcup \{ \sigma^{-n}(P')\st n \geq 0 \}$.  
We thus have
\[ \bigcup \{ \sigma^{-n}(P') \st n \geq 0 \} \subseteq \Omega =   \bigcup \{ \tau^{-1} \sigma^{-n}(P') \st n \geq 0\} =  \bigcup \{  \sigma^{-n} \tau^{-1}(P') \st n \geq 0\},\]
by \eqref{stcommute}.  
Thus  $\tau$ permutes the finitely many orbits of points in $P'$, and thus there is some $k<0$ so that  
\[  \bigcup \{  \sigma^{-n} \tau^{-1}(P') \st n \geq 0\} = \Omega \subseteq  \bigcup \{  \sigma^{-n} (P') \st n \geq k\}. \]  
We observed already that  all points in $ P'$ have critically dense orbits.\ff{In fact, for appropriate $m, n$ we have 
$ \sigma^n = \tau^m$.  To see this, let $m\neq 0$ be such that $\tau^m$ leaves some orbit of a point in $P$ invariant.  Then  for some $p \in P$ and $n \in \ZZ$, we have $\tau^m(p) = \sigma^n(p)$.  Since the fixed point set of $\sigma^n \tau^{-m}$ is Zariski-closed in $X$ and includes all points in the (dense) orbit of $P$, it must consist of all of $X$.}
\end{proof}

Theorem~\ref{thm-main} is immediate from Theorem~\ref{thm-final}.  
We note that we do not believe that the local factoriality hypothesis in Theorem~\ref{thm-main} is necessary.

To end, we comment on some open problems.  
The results in this paper naturally lead to the question of understanding the canonical  birationally commutative factors of noetherian algebras:  that is, of classifying noetherian birationally commutative graded algebras generated in degree 1.  
We  believe that Theorem~\ref{thm-final} should have an extension to any situation where there is a coarse moduli scheme for tails of point modules.  In particular, we conjecture that given a noetherian connected graded $\kk$-algebra $R$, generated in degree 1, and  a  projective scheme $X$ that corepresents tails of $R$-points, the image of $R$ inside $\kk(X)[t; \sigma]$ is in fact contained in the  twisted homogeneous coordinate ring of a $\sigma$-ample invertible sheaf on $X$. We cannot prove this, however; moreover, understanding the image of $R$  would require understanding the analogues of \nba s at positive-dimensional subschemes, and these have not yet been studied.

More interesting, perhaps, is the existence of noetherian birationally commutative algebras with no coarse moduli scheme of tails of points.  The algebra discussed in \cite{GK4} is of this type, as Rogalski and the second author will show in forthcoming work.   Thus algebras such as those appearing in \cite{GK4} must appear in any putative classification of canonical birationally commutative factors.  Note this algebra has no nonzero maps to a twisted homogeneous coordinate ring!  We cannot yet precisely define what geometry characterizes such algebras; these questions are the subject of ongoing research.

Finally, we note that if $P$ is supported on infinite but not critically dense $\sigma$-orbits, then $R(X, \sL, \sigma, P)$ is not noetherian and neither (given sufficient ampleness of $\sL$) is its point space.  This contrasts with the conclusion of Theorem~\ref{thm-main} that $Y_\infty$ is noetherian.  Thus the results here give some evidence for the point of view that noetherianness can be detected geometrically.  We conjecture that the point space $Y_\infty$ of any noetherian algebra is noetherian.

\bibliographystyle{amsalpha}
\def\cprime{$'$}
\providecommand{\bysame}{\leavevmode\hbox to3em{\hrulefill}\thinspace}
\providecommand{\MR}{\relax\ifhmode\unskip\space\fi MR }
\providecommand{\MRhref}[2]{%
  \href{http://www.ams.org/mathscinet-getitem?mr=#1}{#2}
}
\providecommand{\href}[2]{#2}


\end{document}